\documentclass[11pt]{article}

\usepackage[top=2cm, bottom=2cm, left=3cm, right=2cm]{geometry}            

\usepackage{setspace}
\usepackage{graphicx} 
\usepackage{array}
\usepackage{amssymb} 
\usepackage{amsmath}
\usepackage{amsfonts}
\usepackage{amsthm}
\usepackage{multirow}  
\usepackage{array} 
\usepackage{color}  
\usepackage{algorithmic} 
\usepackage[normalem]{ulem}
\usepackage{verbatim}
\usepackage{float}
\usepackage{paralist}  
\usepackage{hyperref}

\newtheorem{theorem}{Theorem}[section]
\newtheorem{assumption}{Assumption}
\newtheorem{proposition}[theorem]{Proposition}

\newtheorem{corollary}[theorem]{Corollary}
\newtheorem{definition}[theorem]{Definition}

\theoremstyle{definition}

\newcommand{\bu}{\mathbf{u}}
\newcommand{\bp}{\mathbf{p}}

\newcommand{\dx}{\dot{x}}

\newcommand{\dq}{\dot{q}}

\newcommand{\Tt}{^{{\mbox{\tiny \bf \sf T}}}}

\newcommand{\J}{{\cal{J}}}
\newcommand{\K}{{\cal{K}}}

\newcommand{\M}{{\cal{M}}}
\newcommand{\B}{{\cal{B}}}
\newcommand{\F}{{\cal{F}}}
\renewcommand{\L}{{\cal{L}}}

\newcommand{\V}{{\cal{V}}}
\newcommand{\W}{{\cal{W}}}
\newcommand{\D}{{\cal{D}}}

\newcommand{\bx}{{\mathbf{x}}}
\newcommand{\bq}{{\mathbf{q}}}

\newcommand{\bv}{{\mathbf{v}}}
\newcommand{\bw}{{\mathbf{w}}}
\newcommand{\bH}{H}
\newcommand{\bB}{{\mathbf{B}}}
\newcommand{\bN}{{\mathbf{N}}}

\newcommand{\bR}{{\mathbf{R}}}
\newcommand{\bA}{{\mathbf{A}}}

\newcommand{\tq}{{\tilde{q}}}
\newcommand{\tP}{{\tilde{P}}}
\newcommand{\tU}{{\tilde{U}}}

\newcommand{\tx}{{\tilde{x}}}
\newcommand{\tcK}{{\tilde{\K}}}
\newcommand{\tcJ}{{\tilde{\J}}}

\newcommand{\hbA}{{\hat{\bA}}}

\newcommand{\setZ}{\mathbb{Z}}
\newcommand{\setU}{\mathbb{U}}
\newcommand{\setP}{\mathbb{P}}

\newcommand{\otheta}{\overline{\theta}}

\DeclareMathOperator*{\argmax}{arg\,max}

\expandafter\def\expandafter\normalsize\expandafter{%
    \normalsize
    \setlength\abovedisplayskip{5pt}
    \setlength\belowdisplayskip{5pt}
    \setlength\abovedisplayshortskip{5pt}
    \setlength\belowdisplayshortskip{5pt}
}

\newcommand{\vL}{\cal{\L}}

\begin{document}
\title{A simplex-type algorithm for continuous linear programs with constant coefficients.\thanks{
Research supported in part by Israel Science Foundation Grant and 711/09  and 286/13.
}}

%

\author{Evgeny Shindin\thanks{
IBM Research,
 Haifa,
Mount Carmel 31905, Israel; 
email  {\tt evgensh@il.ibm.com}.
Research supported in part by
Israel Science Foundation Grant and 711/09  and 286/13.
}  \and
Gideon Weiss \thanks{
Department of Statistics,
The University of Haifa,
Mount Carmel 31905, Israel; 
email   {\tt gweiss@stat.haifa.ac.il}.
Research supported in part by
Israel Science Foundation Grant and 711/09  and 286/13.
}
}


\date{November 15,  2018}  

\maketitle
\begin{abstract}
We consider continuous linear programs over a continuous finite time horizon $T$, with a constant coefficient matrix, linear right hand side functions and linear cost coefficient functions.  Specifically,  we search for optimal solutions in the space of measures or of functions of bounded variation.  
These models generalize the separated continuous linear programming models and their various duals, as formulated in the past by Anderson, by Pullan, and by Weiss. In previous papers we formulated a symmetric dual and have shown strong duality. We also have presented a detailed description of optimal solutions and have defined a combinatorial analogue to basic solutions of standard LP. In this paper we present an algorithm
which solves this class of problems in a finite bounded number of steps, using an analogue of the simplex method, in the space of measures.
\end{abstract}

\medskip

\section{Introduction}
\label{sec.introalgorithm}
This paper presents a finite, exact, simplex-type algorithm for the solution of the symmetric pair of dual continuous linear programs of the form:
\begin{eqnarray}
\label{eqn.mpclp}
&\max & \int_{0-}^T (\gamma+ (T-t)c)\Tt dU(t) \nonumber\\
\mbox{M-CLP}\quad & \mbox{s.t.} & \qquad U(0-) = 0 \nonumber\\
 && \qquad A\, U(t) + x(t) \quad = \beta + b t, 
\quad 0 \le t \le T, \\
 &&  U(t)\ge 0 \mbox{ non-decreasing right continuous,}\; x(t)\ge 0,\; t\in[0,T].\nonumber 
\end{eqnarray}
\begin{eqnarray}
\label{eqn.mdclp}
&\min & \;  \int_{0-}^T (\beta+(T-t)b)\Tt dP(t)  \nonumber \\
\mbox{M-CLP$^*$}\quad  & \mbox{s.t.} & \qquad P(0-) = 0 \nonumber\\
&& \qquad  A\Tt P(t) - q(t) \quad = \gamma + c t,  
\quad 0 \le t  \le T, \\
 &&  P(t)\ge 0 \mbox{ non-decreasing right continuous,}\; q(t)\ge 0,\; t\in[0,T].\nonumber 
\end{eqnarray}
Here $A$ is a $K\times J$ constant matrix, $\beta,b, \gamma, c$ are constant vectors of corresponding dimensions,  the integrals are Lebesgue-Stieltjes, and  include the jumps at 0  in $U,P$.  The unknowns are vectors of cumulative control functions $U, P$ and vectors of non-negative slack or state functions $x, q$, over the time horizon $[0,T]$.   It is convenient to think of dual time as running backwards, so that $P(T-t)$ is the vector of dual variables that correspond to the constraints of (\ref{eqn.mpclp}) at time $t$, and $U(t)$ correspond to the constraints of  (\ref{eqn.mdclp}) at time $T-t$.  The special feature here is that the controls are measures, or equivalently functions of bounded variation.
We refer to the  coefficients of the matrix $A$ as the structure parameters, to $b,\,c$ as the rate parameters, and to $\beta,\gamma,T$ as the boundary parameters of the M-CLP problem.

This paper continues research described in \cite{weiss:08,shindin-weiss:13,shindin-weiss:14}.
The algorithm for the solution of M-CLP builds on and extends the algorithm \cite{weiss:08} for solution of SCLP, and is somewhat similar in principle.  While some SCLP problems cannot be solved with the algorithm of \cite{weiss:08}, every SCLP problem can be solved as a special case of an M-CLP problem, by the algorithm presented in this paper.
The problem (\ref{eqn.mpclp}) is solved parametrically.  It  starts from an artificial set of boundary parameters $T_0,\beta_0,\gamma_0$ that have a simple optimal solution, and moves in a finite bounded number of steps along a parametric straight line $\L(\theta),\, 0\le \theta \le 1$,  to the boundary parameters $\beta,\gamma,T$ of the original problem, where in each step an M-CLP pivot is performed.  


At this stage it is too early to estimate the efficiency of this algorithm.  Like standard LP simplex its worst case performance is exponential.  It is however worth mentioning here, that 
we are currently in the process of re-programming the SCLP algorithm, beyond the pilot inefficient MATLAB implementation that was reported in \cite{weiss:08}.  While we have not yet stabilized the improved code, we made some improvements to the MATLAB current code.  In comparing our current implementation of the simplex based algorithm \cite{weiss:08}, against an optimized uniform discretization solved by CPLEX we have found that the simplex based  algorithm can outperform CPLEX both in a run-time (more then 10 times faster for some problems) and in the quality of the obtained solutions.  Moreover we have found that the  simplex based algorithm is able to solve very large problems (10000 controls and 10000 constraints), and these problems  can only be solved by CPLEX with very coarse discretization.
Based on these encouraging results, we expect that the M-CLP simplex based algorithm described in this paper, will be competitive with the discretization approach.

\subsection{Background and motivation}
These problems belong to the area of continuous linear programming (CLP), as pioneered by Bellman in 1953, and further discussed and investigated in \cite{grinold:70,levinson:66,tyndall:65,tyndall:67}.  Bellman, was motivated by economic problems, in particular the optimization of continuous time Leontief input output economics, which also motivated Dantzig \cite{dantzig:51}.  No direct solution method for these continuous time problems has emerged, but it was immediately observed that by discretizing time these problems could be approximated by standard linear programs.  Among others, discrete time Leontief models motivated the Dantzig-Wolfe decomposition method \cite{dantzig-wolfe:60}.   
In 1978 Anderson \cite{anderson:81} formulated the less general separated continuous linear programming  problem (SCLP), motivated by some job shop scheduling problems.  Further work by Anderson and co-workers included \cite{anderson-philpott:89} and a book \cite{anderson-nash:87}, and much progress was achieved by Pullan \cite{pullan:93,pullan:95,pullan:96,pullan:97,pullan:00}, who formulated a non-symmetric dual, proved strong duality, and devised an algorithm based on a sequence of discretizations that converged to optimality.  SCLP found several important applications, including beside job-shop scheduling also control of queueing networks \cite{nazarathy-weiss:09}, quickest transshipment problem\cite{hoppe-tardos:00}, and general network flows over time \cite{fleischer-tardos:98,fleischer-skutella:07}.   However, the solution methods for these problems remained by discretization of time.  Several ingenious methods for efficient approximations to solutions of SCLP have appeared:  
Fleischer-Sethuraman \cite{fleischer-sethuraman:05} use geometrically scaled discretization, Bertsimas-Luo \cite{bertsimas-luo:98} use a non-convex quadratic programming formulation, and Bampou-Kuhn \cite{bampou-kuhn:12} use  polynomial approximations.

The first algorithm to provide an exact  finite solution method for SCLP was devised by Weiss \cite{weiss:08} who considered the following version of SCLP:
\begin{eqnarray}
\label{eqn.psclp}
&\max & \int_0^T ((\gamma\Tt+(T-t) c\Tt) u(t) + d\Tt x(t))\, dt \nonumber\\
\mbox{SCLP    }\quad \quad &\mbox{s.t.} & \int_0^t G u(s) ds +F x(t) \le
\alpha + a t, \\
&&  \quad  H u(t)  \quad \quad \quad \quad \le b,
\nonumber
\\ &&  x(t),u(t) \ge 0, \quad t \in [0,T]  \nonumber.
\end{eqnarray}
where $G,\,H,\,F$ are fixed 
matrices, $\alpha,\,a,\,\gamma,\,c,\,b,\,\,d$ are fixed vectors, and the  unknowns  are bounded measurable controls $u_j(t)$ and absolutely continuous states $x_k(t)$.
For this problem Weiss has  formulated a symmetric dual problem with bounded measurable controls $p_k(t)$ and  states $q_j(t)$,  found a combinatorial description of the extreme points of the feasible regions of the primal and dual, and devised a simplex type algorithm that solves these problems in a finite bounded number of steps, all of this under some sufficient conditions.  
We note that many of the applications of SCLP can indeed be formulated as (\ref{eqn.psclp}).

The results of \cite{weiss:08} left a large gap in the theory, since there is no strong duality between the SCLP problem (\ref{eqn.psclp}) and its symmetric dual.  Therefore it is possible  that SCLP (\ref{eqn.psclp}), or its dual, or both, have no optimal solution even if both are  feasible.  In any of these cases, the algorithm devised by Weiss  fails to work.  

The main idea in \cite{shindin-weiss:13,shindin-weiss:14} and in the current paper, is to formulate the problems M-CLP/M-CLP$^*$ (\ref{eqn.mpclp}), (\ref{eqn.mdclp}) in the space of measures, or equivalently in the space of functions of bounded variation.   Already Pullan's dual was formulated in the space of measures, but it was not a symmetric dual, and could not be used in a simplex type algorithm.  The only other paper that we know of that looks for solutions in the space of measures is Papageorgiou \cite{papageorgiou:82}.  The  formulation of M-CLP/M-CLP$^*$  as (\ref{eqn.mpclp}), (\ref{eqn.mdclp}),  with the unknowns as functions of bounded variation, turns out to {\em achieve complete analogy with standard, finite dimensional linear programs (LP)}.  It is shown in \cite{shindin-weiss:13,shindin-weiss:14} that strong duality holds, with a complete combinatorial description of the finite number of extreme points of the feasible region.  The last step in the analogy with  LP is a finite simplex type algorithm derived in this paper.  Here M-CLP stands for CLP in the space of measures.

\subsection{Overview and contribution}
At this point it is useful to consider the differences between SCLP (\ref{eqn.psclp}), and M-CLP (\ref{eqn.mpclp}).  
Optimal solutions of SCLP have piecewise constant control rates $u_j(t)$ and continuous piecewise linear states $x_k(t)$, with complementary slack dual controls $p_k(T-t)$ and dual states $q_j(T-t)$.  Hence the solution can be summarized by the breakpoints 
$0=t_0<\cdots<t_N=T$, by the values of $u_j(t)$ and of the derivatives $\dx_k(t)$ and their duals in each interval, and by the boundary values of the primal $x_k(0)$, and dual $q_j(0)$.  
These are determined as follows:  The boundary values $x_k(0)$ and $q_j(0)$ are determined by two standard LPs, with coefficient matrices $F$ and $H$ respectively.  The rates $u_j,\dx_k,p_k,\dq_j$ in each interval are a pair of complementary slack basic solution of an LP (referred to as the Rates-LP) with coefficient matrix $\left[\begin{array}{cc}G & F \\ H & 0\end{array}\right]$ (replaced by $A$ in M-CLP), and each of the breakpoints is determined  by an equation of the form $x_k(t_n)=0$ or of the form $q_j(T-t_n)=0$.  It follows from this structure that the optimal solution to an SCLP is determined by an optimal sequence of adjacent (differ by a single LP pivot) bases $B_1,\ldots,B_N$ of the Rates-LP, and given this sequence the optimal solution can be computed directly.  
The SCLP algorithm constructs this optimal sequence of bases.  

The algorithm solves SCLP  parametrically, over the parametric line of time horizons moving from $0$ to $T$.  This is similar to the parametric self dual simplex algorithm, also known as Lemke's algorithm \cite{lemke:65,vanderbei:14}.
As a first step, the vectors of boundary values of the optimal solution, $x(0),\,q(0)$ are calculated, and these are valid and remain fixed for all time horizons.  The algorithm moves from 0 to $T$ in a finite bounded number of steps, at each of these steps an SCLP pivot is performed, to update the sequence of bases.  In this SCLP pivot some bases are deleted and some bases are added to the sequence of bases.

The optimal solution of M-CLP in the open interval $t\in (0,T)$ is likewise  given by continuous piecewise linear $U_j(t)$ with  piecewise constant derivatives $u_j(t)$ and continuous piecewise linear $x_k(t)$ with complementary slack duals $P_k(T-t)$ and $q_j(T-t)$, which are again determined by a sequence of adjacent bases $B_1,\ldots,B_N$ of a Rates-LP with coefficient matrix $A$.  In addition the solution  includes impulse controls at 0 and $T$.  Thus, the boundary values of optimal M-CLP/M-CLP$^*$  solutions consist of primal and dual impulse controls $U(0)$, $U(T)-U(T-)$, and $P(0)$, $P(T)-P(T-)$, and of 
the jumps in the states,  $x_k(0)$, $x_k(T)-x_k(T-)$, $q_j(0)$, $q_j(T)-q_j(T-)$ (here we define by convention $x(0-)=0,\, q(0-)=0$). 
These results were derived in  \cite{shindin-weiss:13,shindin-weiss:14}.  
The solution of M-CLP can be decomposed into a solution for an Internal-SCLP problem, solved by a sequence of adjacent  bases $B_1,\ldots,B_N$ of $[A\;I]$ (which can be solved by the algorithm of \cite{weiss:08}), and a pair of Boundary-LP problems.
One of the boundary LPs determines the  optimal primal boundary values at $0$ and $T$, the other determines the optimal dual boundary values at $0$ and $T$.  The two boundary problems are of dimension $2K\times 2J$ each.  They share the same (transposed) matrix  of coefficients, and their solutions are complementary slack but not dual.

The algorithm for M-CLP uses a more complicated parametric line $\L(\theta)$, $0\le \theta \le 1$, that moves from initial values $\beta_0,\gamma_0,\lambda_0,\mu_0,T_0$ to the final values $\beta,\gamma,0,0,T$ ($\lambda,\,\mu$ are additional vectors of boundary parameters  needed here, but they disappear in the final solution, they will be explained in Section \ref{sec.structure} and Appendix \ref{sec.lambdamu}).  The solution again moves along the parametric line in a finite number of steps, where at each step an M-CLP pivot is performed.  This M-CLP pivot may be an internal pivot on the Internal-SCLP or it may be a boundary pivot, that involves the Boundary-LPs.  

Two issues had to be resolved in the derivation of the M-CLP algorithm.
The decomposition into Boundary-LPs and Internal-SCLP is quite subtle, it will be derived in Section \ref{sec.decomp}.  
Also,
a significant difference between the SCLP and the M-CLP algorithm is that for SCLP the equations for the boundary values, for the time intervals, and for the values of the state variables at the breakpoints are decoupled, while in the M-CLP problem these equations are coupled.  These equations need to be solved  to obtain the next point of the parametric line at which a constraint becomes tight and the next iteration is performed, as explained in Section \ref{sec.algorithm_simpl}.

\subsection{Applications and extensions}
We now list some potential applications and possible extensions of M-CLP.

\subsubsection*{Solution of SCLP porblems}  By formulating SCLP as M-CLP, problems that cannot be solved by the algorithm of \cite{weiss:08} can now be solved.  

\subsubsection*{Sensitivity analysis:}
The fact that we obtain an exact solution, with a complete  detailed structure of the solution, enables us to do sensitivity analysis similar to what is done with standard LP.  For example, we can see how the solution may change if we change the time horizon, or if we move some of the time points $t_n$.  This is much harder to do if the problem is solved by discretization.

\subsubsection*{Rolling horizon and model predictive control}  In many applications of optimization over time, problems are solved for some given time horizon $T$, and are then re-solved repeatedly, for the same length of time horizon, after every $d=T/m$ time units, with updated data.  For M-CLP that would mean updated $\beta, \gamma$.  By using the simplex type algorithm for M-CLP, this can be done very efficiently, since updating requires only a fraction of the calculations required to solve the complete problem.  This efficiency is lost if one uses discretization.

\subsubsection*{Continuous time Leontief economic model}  We can formulate and solve Leontief economic model problems in continuous time, using an M-CLP formulation, if the target capacity function is given in advance (cf. \cite{shindin:16} Section 1.3).  If one wishes also to optimize the capacity this becomes a CLP of the type formulated by Bellman \cite{bellman:53} which is not an M-CLP.  

\subsubsection*{SCLP and M-CLP with piecewise constant data}  The M-CLP algorithm can  directly be generalized to solve problems with piecewise constant $b,c,\beta,\gamma$.  Solutions of such problems will have impulse controls at points at which the data changes.

\subsubsection*{Fractional continuous linear programs} 
Problems with linear constraints and an objective function that is the ratio of two linear terms are known as fractional linear programs, and can be solved by a simplex algorithm.  A continuous time version of this was studied by Wen \cite{wen:13a,wen:13b}, who analyzed their properties and suggested an approximation algorithm that included bounding the objective and discretizing time. 
It is possible to formulate such problems as M-CLP, and use our algorithm to obtain the exact solution in a finite number of steps.

\subsubsection*{Continuous linear complementarity problems} 
Linear complementarity problems (LCP) have been studied since the 60's (cf. comprehensive book of Cottle, Pang and Stone \cite{cottle-pang-stone:92}).  While general  LCP are NP-hard, several types of LCP's can be solved by a simplex type algorithm developed by Lemke and Howson  \cite{lemke:62,lemke:65,lemke-howson:64}.  These include convex quadratic programming problems, and bi-matrix games.  A continuous  time version of LCP can be formulated, and  a simplex type algorithm that is similar to our M-CLP algorithm can be used to solve these problems.

\subsubsection*{More general right hand side and objective functions}  Pullan \cite{pullan:95} has shown that if the right hand side and objective functions are piecewise analytic, and the feasible region is bounded, SCLP has a piecewise analytic optimal solution.  We hope to obtain similar results for M-CLP, and extend the algorithm to such problems.  It is reasonable to conjecture that optimal solutions will again consist of a  finite partition of the time horizon, where each interval is associated with a basis of $[A\; I]$, and the control rates and rates of the states are analytic functions in each interval.  It seems that one can then extend our algorithm to such r.h.s. and objective, with the added difficulty that along the parametric line, constraints become  tight not just at the breakpoints, but also at internal points of  intervals, at times at which the derivative of a state variable is 0 (i.e. at local minima).

\subsection{Outline or the rest of the paper}
The rest of the paper is constructed as follows:  In Section \ref{sec.example} we give an example of an M-CLP problem solved by our algorithm.  It is the smallest possible example, but was chosen to have enough iterations to illustrate most features of the algorithm. This should provide some intuition,  before we introduce the actual algorithm in Section \ref{sec.algorithm_simpl}.   We return to this example in Section  \ref{sec.example2}.

 In the next part of the paper we extend some of the results of \cite{shindin-weiss:14}:  we give a succinct description of the solution (Section \ref{sec.structure}),  then describe a decomposition of M-CLP and its dual to an Internal-SCLP problem and a pair of Boundary-LP problems (Section \ref{sec.decomp}), and follow with a discussion of validity regions (Section \ref{sec.region}).  

To construct the algorithm, we discuss in Section  \ref{sec.collisions} what happens on the boundaries of validity regions.  When $\L(\theta)$ approaches the boundary of a validity region, at the boundary, some of the values of the impulse controls or the state variables or the interval lengths, which are positive in the interior of the validity region, reach 0.  As a result,  the set of positive boundary values and  sequence of internal bases is no longer  optimal outside the boundary.  We call this a collision, and  we classify all possible types of collisions.   

In Section \ref{sec.pivot}  we discuss M-CLP pivots.  
We consider two neighboring  validity regions $\V,\W$, and given the type of collision as $\L(\theta)$ approaches the boundary of $\V$, we show how to identify the collision as   $\L(\theta)$ approaches the point of collision from the other direction, i.e. from within $\W$.  This enables us to identify the set of positive boundary values and the  sequence of internal bases that are optimal in $\W$, which completes the pivot.

Based on the pivot operation we then describe the algorithm in Section \ref{sec.algorithm_simpl}, where we show how to determine the collision points and perform the sequence of pivots.       In Section \ref{sec.example2} we return to the example of Section \ref{sec.example}, and re-interpret it in terms of the algorithm. 

The algorithm of Section \ref{sec.algorithm_simpl} requires us to make some strong simplifying assumptions in order for it to function.  However, in all cases these assumptions can be enforced by perturbation of the parametric line.  We outline the necessary perturbation procedures in Appendix \ref{sec.general}.  

For ease of reading and continuity we moved some of the proofs of Sections \ref{sec.region}, \ref{sec.pivot} and \ref{sec.algorithm_simpl} to  Appendix \ref{sec.validityappendix}, \ref{sec.uniqueproof}, \ref{sec.alg-line} respectively.



{\em Notes on presentation:} Throughout the paper we quote results from \cite{shindin-weiss:13,shindin-weiss:14}.  We will refer to results from the duality paper \cite{shindin-weiss:13} by appending D and from the structure paper \cite{shindin-weiss:14} by appending S (e.g. D5.3, or (S4.2)).
There is complete symmetry between M-CLP and M-CLP$^*$.  Therefore when we formulate results for one or both of them, we will give a proof only for one of them.


\section{An Illustrative  Example}
\label{sec.example}
We now apply the algorithm to solve the following problem:
\begin{eqnarray}
\label{eqn.example}
&\max & \;   \int_{0-}^3 \left(\big[-4; -2\big] + (3-t)\big[3; 2\big] \right)
\left[ \begin{array}{c} dU_1(t) \\ dU_2(t) \end{array} \right] 
 \nonumber\\
& \mbox{s.t.} & 
\left[ \begin{array}{cc} 2 & 1 \\ 1 & 1  \end{array} \right] 
\left[ \begin{array}{c} U_1(t) \\ U_2(t) \end{array} \right]
+ \left[ \begin{array}{c} x_1(t) \\ x_2(t) \end{array} \right]
= \left[ \begin{array}{c} 4\\ 1 \end{array} \right]
 + \left[ \begin{array}{c} 4\\ 3 \end{array} \right]  t, 
 \\
 && \; U(t) \ge 0 \mbox{ non-decreasing, }\; x(t)\ge 0,\quad 0\le t \le 3.\nonumber 
\end{eqnarray}


As can be seen by comparing with (\ref{eqn.mpclp}),  the data of this problem is:
\begin{equation}
\label{eqn.exampledata}
\begin{array}{clll}  
\gamma = [ -4; -2], \\
c  =  [ 3; 2], \\
A  =  \left[ \begin{array}{cc} 2 & 1 \\ 1 & 1  \end{array} \right],  &
b  = \left[ \begin{array}{c} 4\\ 3 \end{array} \right], &
\beta =   \left[ \begin{array}{c} 4\\ 1 \end{array} \right], &  T=3.
\end{array}
\end{equation}

Figure \ref{fig.p2full} gives a graphical representation of the complete solution of this problem, as well as of its dual.   
The challenge in presenting this solution in a graph is that we wish to show 8 different function variables in one picture, in addition to impulse controls at 0 and $T$.  We also wish to distinguish the different types of variables and  illustrate the duality.
To do so we use the horizontal axis of length $T$ to present primal time from left to right and dual time in the reversed direction from right to left.  We present primal states with the complementary slack dual control rates in the upper quadrant, and the dual states with the complementary slack primal control rates  in the lower quadrant.  To be precise, we plot all the the non-zero values of:  $x_k(t)$, $p_k(T-t)$, 
$-q_j(T-t)$, and $-u_j(t)$.  
We do not plot 0 values, and because basic solutions of the $2\times 2$ Rates-LP have only 2 non-zero primal values and   2 non-zero dual values,  we only need to plot two positive and two negative values for each $0<t<T$. 
It is easy to distinguish between state variables ($x_k,\,q_j$) which are continuous piecewise linear, and control rates (derivatives of the cumulative controls, $u_j,\,p_k$) which are piecewise constant.

The impulse controls are shown as vertical arrows, next to the origin and next to the horizontal coordinate $T$.  
Next to the origin we plot  $-U_j(0)$ and  $P_k(T)-P_k(T-)$.  Next to  $T$ we plot $P_k(0)$ and 
$-(U_j(T)-U_j(T-)$.  Circles are used at the origin for  jumps in $q_j(T)$, and at $T$ for jumps in $x_k(T)$. 

\begin{figure}[h!t] 
\centering 
\includegraphics[width=3.3in]{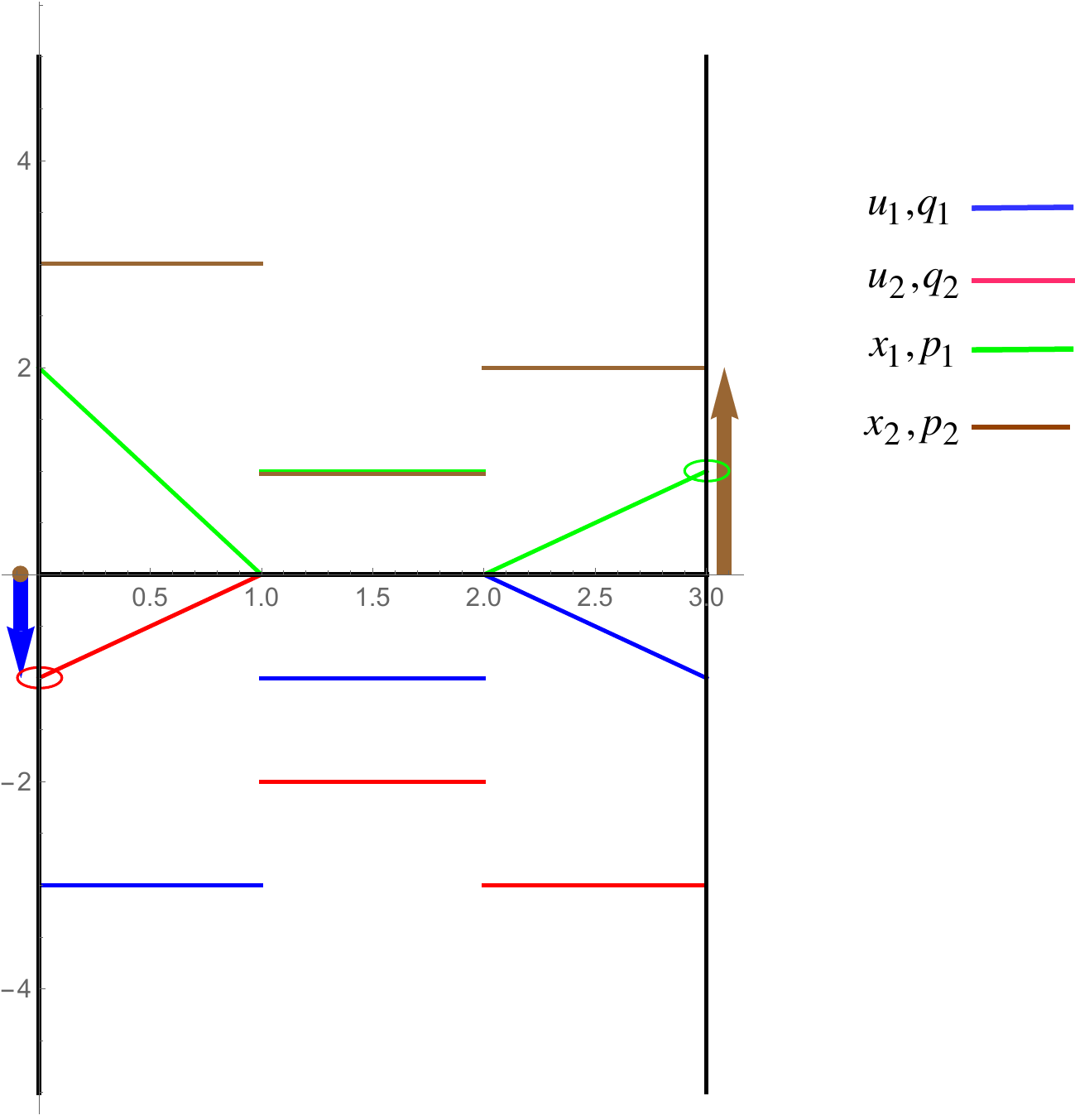} 
\caption{Full Solution of the example problem.}
\label{fig.p2full}
\end{figure}

We use different colors to refer to the primal and dual states, and for each state we use the same color also  for the complementary control.  We use the same colors in all the following figures.

In this figure we can see that the solution consists of 3 intervals, and a single impulse control at 0 for the primal and another at 0 for the dual (in the reversed time). Circles denote the values of the state variables at $T$, which might be different from the  limiting  values at $T{-}$.

To apply the algorithm we will solve the problem along a parametric line, that will move along the {\em line of boundary parameters} $\L(\theta)$ from $\theta=0$ to $\theta=1$, where 
\begin{equation}
\label{eqn.exampleline}
\begin{array}{ccl}
\L(0) =& (\beta_0,\gamma_0,\lambda_0,\mu_0,T_0) &= \left(\left[ \begin{array}{c} 2\\ 1 \end{array} \right], \left[ \begin{array}{c} -5\\ -5 \end{array} \right],\left[ \begin{array}{c} -1\\ -1 \end{array} \right],\left[ \begin{array}{c} 1\\ 1 \end{array} \right],1\right),  \\
\L(1) =& (\beta,\gamma,\lambda,\mu,T) &= \left(\left[ \begin{array}{c} 4\\ 1 \end{array} \right],
\left[ \begin{array}{c} 1\\ 2 \end{array} \right],\left[ \begin{array}{c} 0\\ 0 \end{array} \right],\left[ \begin{array}{c} 0\\ 0 \end{array} \right],3\right),
\end{array}
\end{equation}
The artificial boundary parameters $\lambda,\mu$ are used to avoid degeneracies along $\L(\theta)$, and disappear at $\theta=1$.  They will be explained in Appendix \ref{sec.lambdamu}.

The algorithm will solve this problem in 7 iterations, the problem was deliberately chosen to have many iterations, so as to demonstrate various features of the algorithm.  The iterations happen at 
$\theta$ values $0,\frac{1}{11},\frac{1}{6},\frac{2}{9},\frac{4}{9},\frac{5}{9},\frac{13}{19},1$.
In each of the intervals between these $\theta$s, the optimal solution maintains the same set of positive boundary values, the same number of breakpoints in the solution, and the same rates of controls and slopes of the states between breakpoints.  

As $\theta$ changes, the solution changes continuously in an affine way, until a boundary constraint or an internal constraint becomes tight.  At this $\theta$ the edge of the validity region is reached and the pivot of the iteration happens.
If  a boundary value shrinks to zero then we perform a boundary pivot. Otherwise, if some intervals shrink to zero, or some primal or dual state variable, at a breakpoint, shrinks to zero, we perform an Internal-SCLP pivot.  We call this an M-CLP pivot.  

The initial values $\L(0)$ are chosen so that the solution has a single interval, with 0 controls and states equal to the right hand sides: $x(t) = \beta(\theta)+b t$, $q(t) = -\big( \gamma(\theta) + c t\big)$, $0<t<T(\theta)$, for small $\theta$.  

In the following paragraphs we present illustrations of the optimal solution at each iteration, 
showing the solution at the boundary and the middle of the interval of $\theta$s, and we briefly explain what happens at the pivot.  We return to this example in Section \ref{sec.example2}, with additional details, after the algorithm is derived.

{\bf Iteration 1:}  Starting at $\theta =0$ the solution is valid until $\theta =1/11$, when $q_1(T)$ shrinks to 0  (at the left side of the figure, recall dual time runs right to left). 
\begin{figure}[h!t] 
\centering 
\includegraphics[width=4.5in]{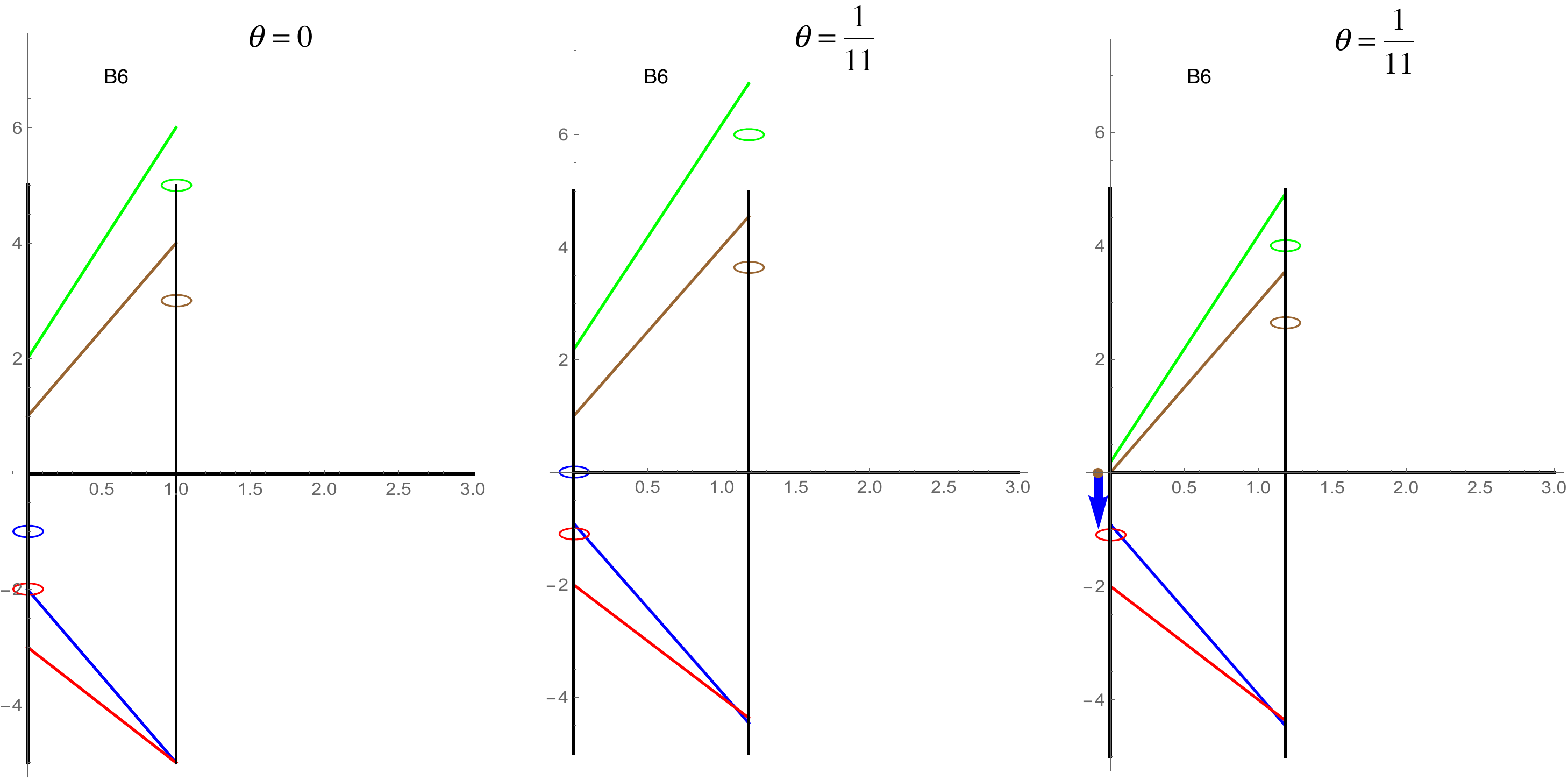} 
\caption{Example problem, Iteration 1, including a boundary pivot.}
\label{fig.p2iteration1}
\end{figure}
A boundary pivot is performed and as a result a primal  impulse control $U_1(0)$ at time 0 is introduced, and the boundary state $x_2(0)$ jumps down to 0.  The boundary pivot also introduces a dual impulse control $P(T)-P(T-)$ that is still 0 but will increase in the next validity region.

{\bf Iteration 2:}  Starting at $\theta =1/11$ the solution is valid until $\theta =1/6$, when $q_1(T-)$ shrinks to zero.  In the figure we show the solution also at $\theta =1/8$.  Note that the impulse control $P(T)-P(T-)$ at the left side of the figure is now positive and growing.  
\begin{figure}[h!t] 
\centering
\includegraphics[width=4.5in]{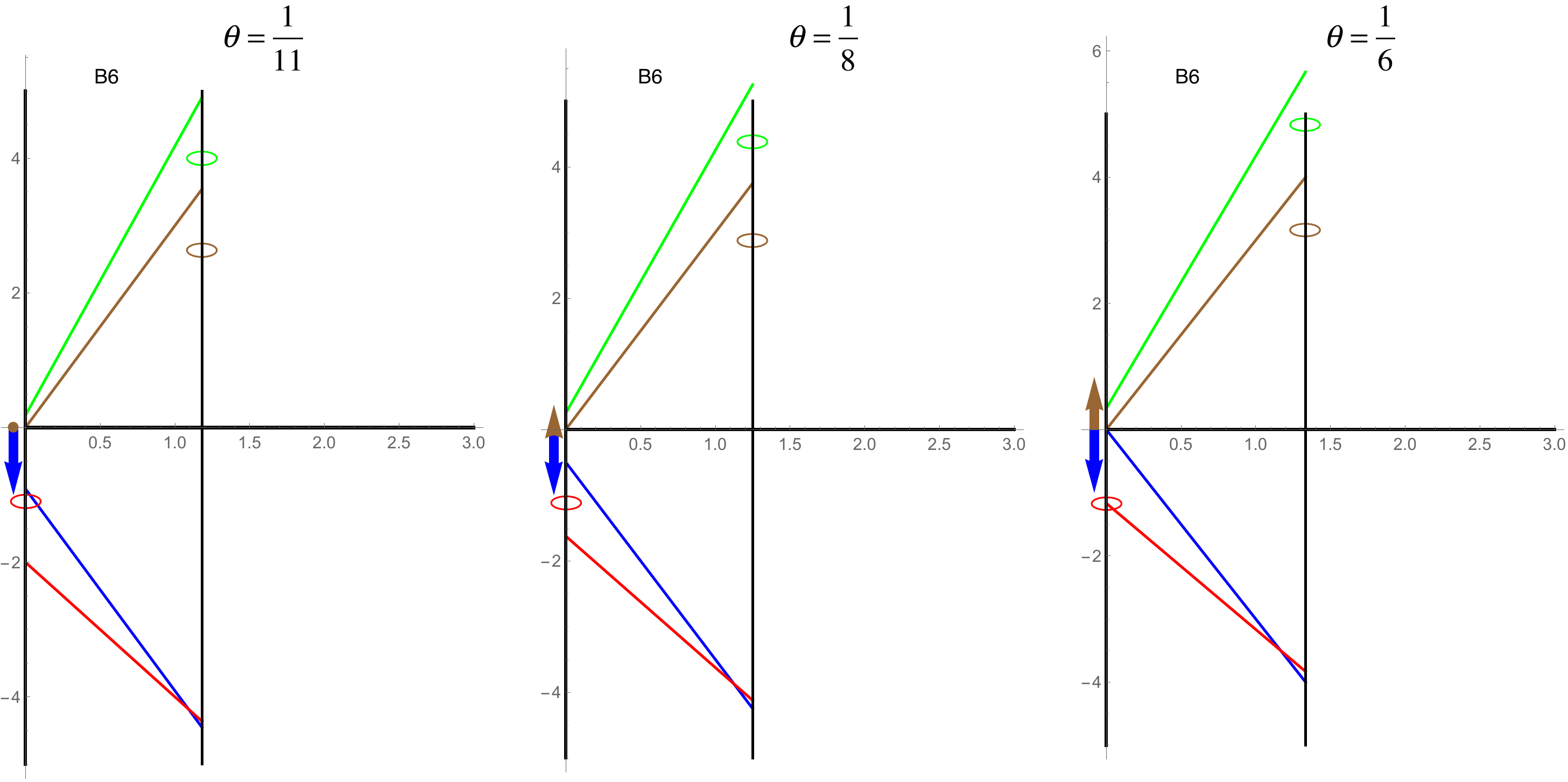}
\caption{Example problem, Iteration 2.}
\label{fig.p2iteration2} 
\end{figure}
An Internal-SCLP pivot is performed, in which a new interval is inserted at primal time 0 (dual time $T$), where the state variable $q_1$  shrunk to 0.   Currently the new interval has length 0.  It will increase with $\theta$.

{\bf Iteration 3:} Starting at $\theta = 1/6$ the solution now consists of two intervals, $0<t_1<T$.  At $\theta = 2/9$ the value of $x_1(t_1)$  shrinks to 0. This happens at $t_1=2/9$, when the local minimum of $x_1(t)$ descends to 0.  We also show the solution at $\theta=7/36$, where the local minimum of $x_1(t)$ is still $>0$.   Note that the new interval has increased in length, notice also that there are  positive primal and dual control rates in this interval.
\begin{figure}[h!t] 
\centering 
\includegraphics[width=4.5in]{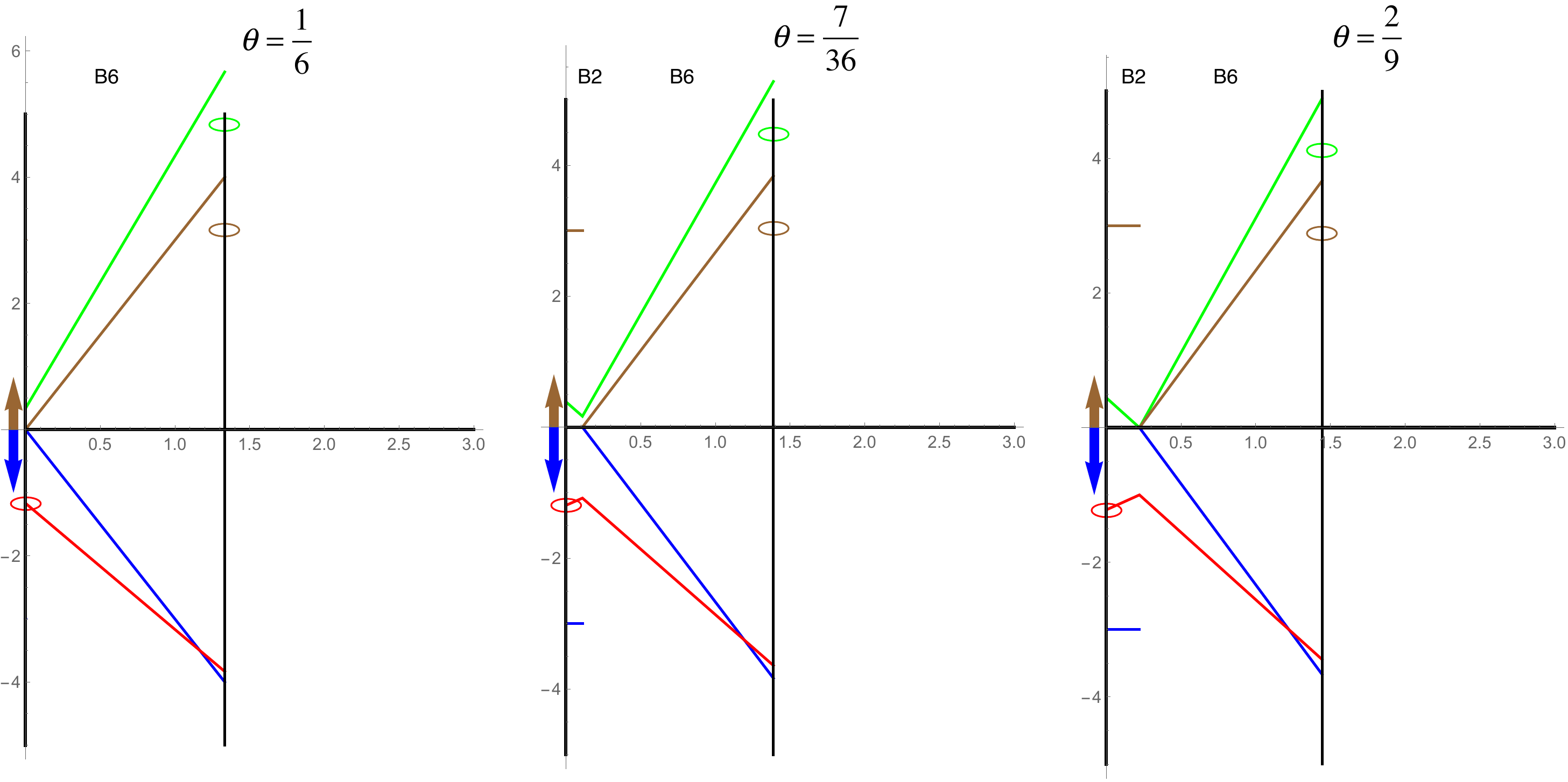} 
\caption{Example problem, Iteration 3.}
\label{fig.p2iteration3}
\end{figure}
An Internal-SCLP pivot is performed, and a new interval is inserted at $t_1$.  Currently the new interval has length 0.

{\bf Iteration 4:} Starting at $\theta = 2/9$ the solution now consists of three intervals, $0<t_1<t_2<T$.  At $\theta = 4/9$ the value of $q_2(T-t_1)$ shrinks to 0.  We also show $\theta = 3/9$.
\begin{figure}[h!t]
\centering 
\includegraphics[width=4.5in]{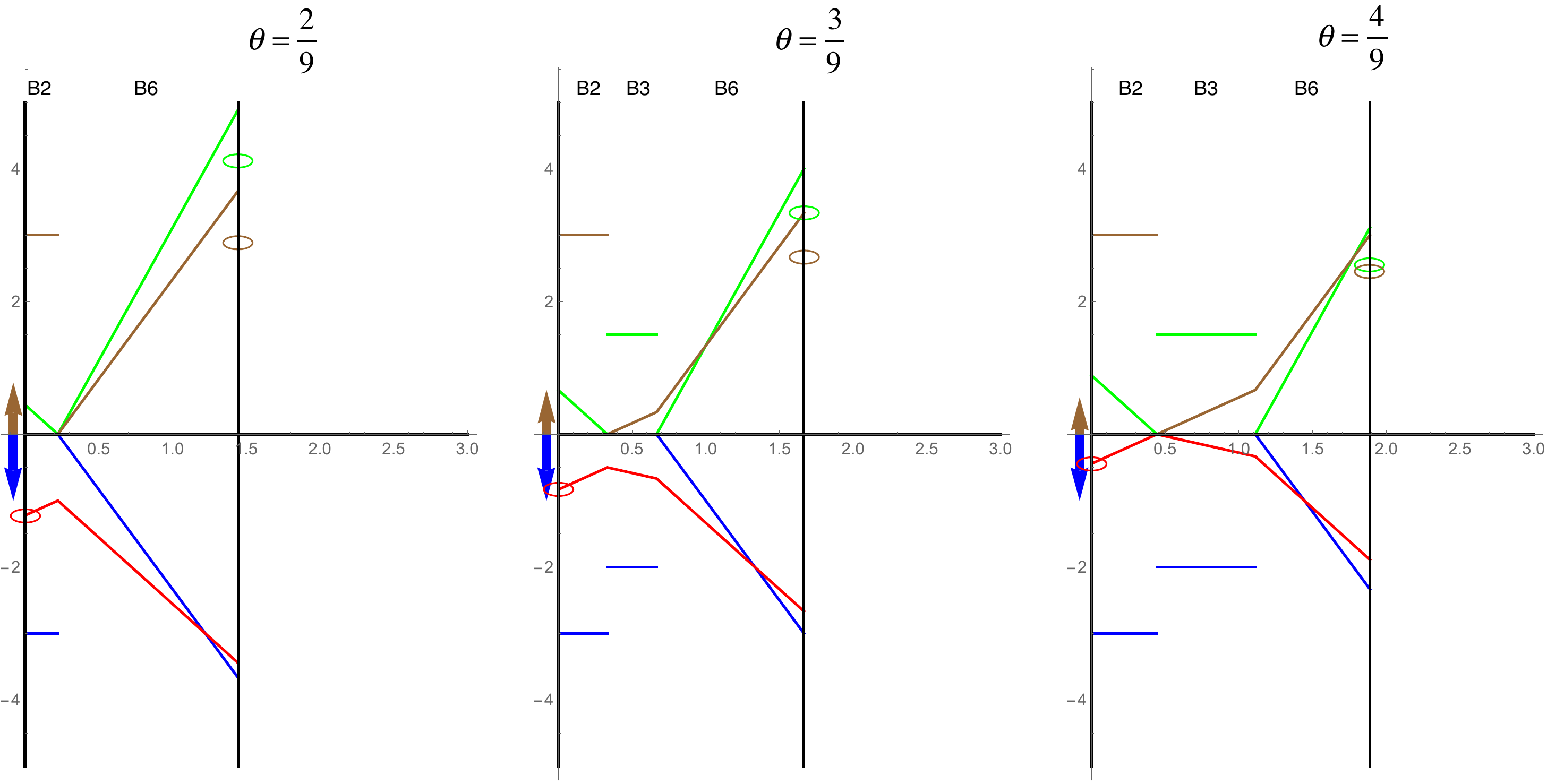}
\caption{Example problem, Iteration 4.}
\label{fig.p2iteration4}
\end{figure}
An Internal-SCLP pivot is performed, and another new interval is inserted at $t_1$.  It is currently of length 0.

{\bf Iteration 5:} Starting at $\theta = 4/9$ the solution now consists of four intervals, $0<t_1<t_2<t_3<T$.  At $\theta = 5/9$ the interval between $t_2$ and $t_3$ shrinks to 0.   We also show $\theta = 1/2$.  Note that in the interval $(t_1,t_2)$ all the states are zero, and all four control rates are positive.
\begin{figure}[h!t] 
\centering
\includegraphics[width=4.5in]{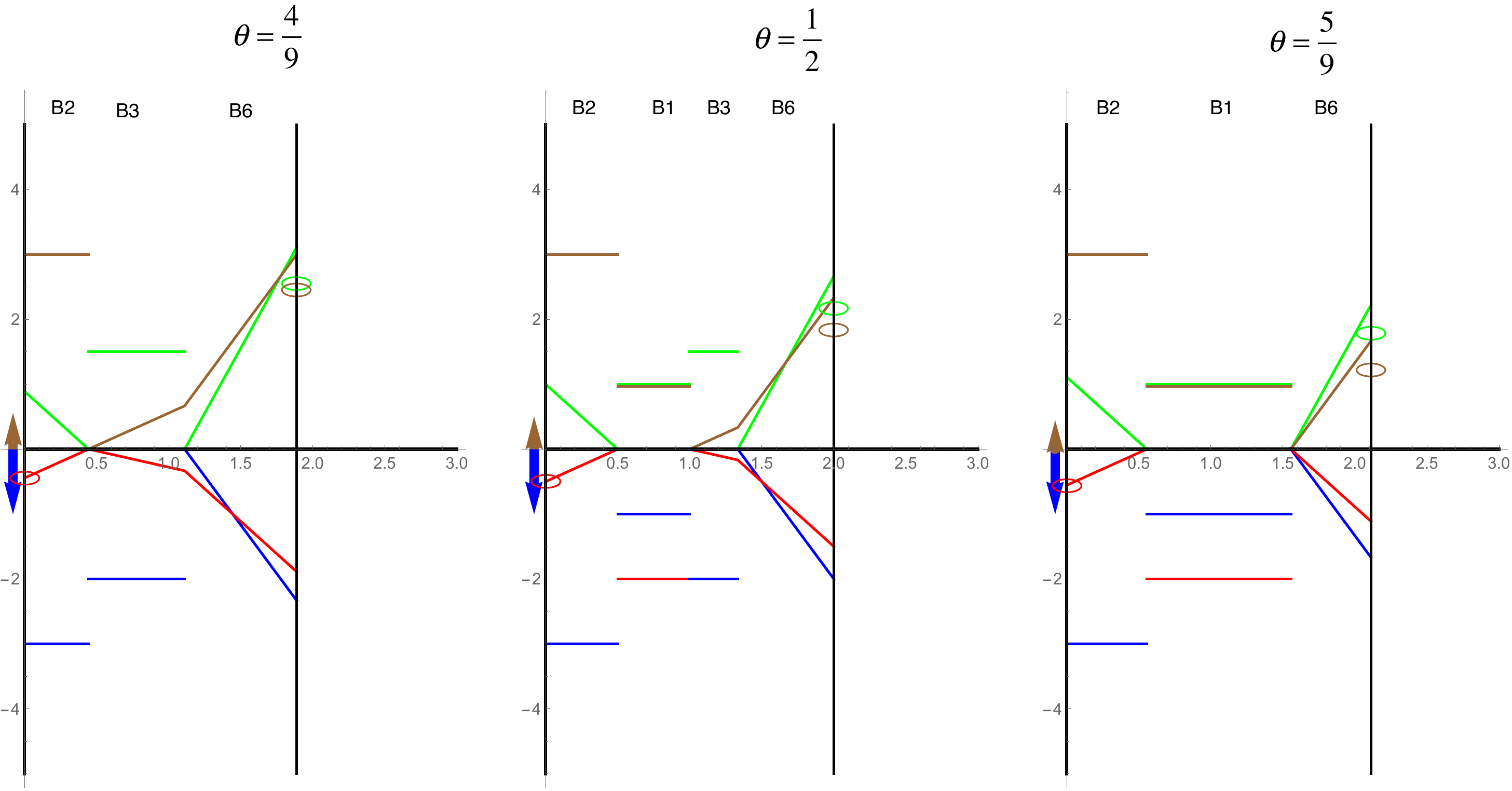}
\caption{Example problem, Iteration 5.}
\label{fig.p2iteration5}
\end{figure}
An Internal-SCLP pivot is performed, and the positive variables  in the 3rd interval are changed.  
What happened was that for $4/9<\theta<5/9$, at $t_2$ we had $x_2(t_2)=q_2(T-t_2)=0$, and  at $t_3$,
$x_1(t_3)=q_1(T-t_3)=0$.  As $\theta \to 5/9$, $t_2$ and $t_3$ moved toward each other until they met.  After the pivot, with $\theta > 5/9$ they continued to move in the same directions, so now the roles of $t_2,t_3$ have switched, and we will have for $\theta > 5/9$ that $x_2(t_3)=q_2(T-t_3)=0$ and 
$x_1(t_2)=q_1(T-t_2)=0$.
So in this pivot, the interval $(t_2,t_3)$ shrunk to zero, and after the pivot it will increase again, but with different positive state variables.

{\bf Iteration 6:} Starting at $\theta = 5/9$ the solution now consists of four intervals, $0<t_1<t_2<t_3<T$, where the interval $(t_2,t_3)$ now has a different basic solution of the Rates-LP.  At $\theta = 13/19$ 
the boundary value of $x_2(T)$ shrinks to 0.  
\begin{figure}[h!t]
\centering
\includegraphics[width=4.5in]{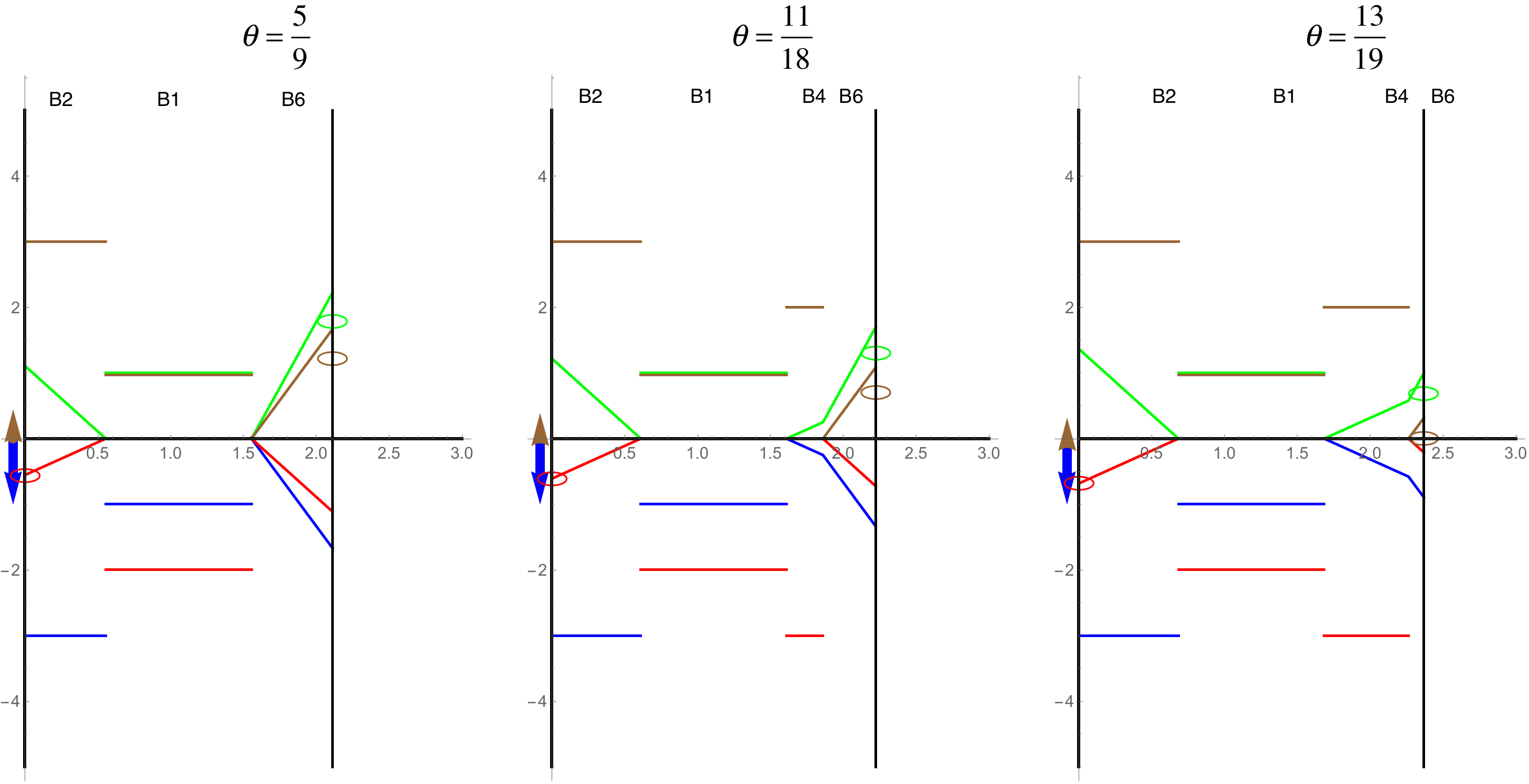}
\caption{Example problem, Iteration 6.}
\label{fig.p2iteration6} 
\end{figure}
A boundary pivot is performed.  A new impulse control $P_2(0)$ is introduced at dual time 0 (primal time $T$, i.e. at the right end of the illustration).  This impulse control is currently still 0. 

{\bf Iteration 7:} Starting at $\theta = 13/19$ the solution still consists of four intervals.  At $\theta =1$, the interval between $t_3$ and $t_4$ shrinks to 0, and only three intervals remain.  Also the impulse control $P_2(T)-P(T-)$, at the left end of the illustration, shrinks to 0.  
\begin{figure}[h!t]
\centering
\includegraphics[width=4.5in]{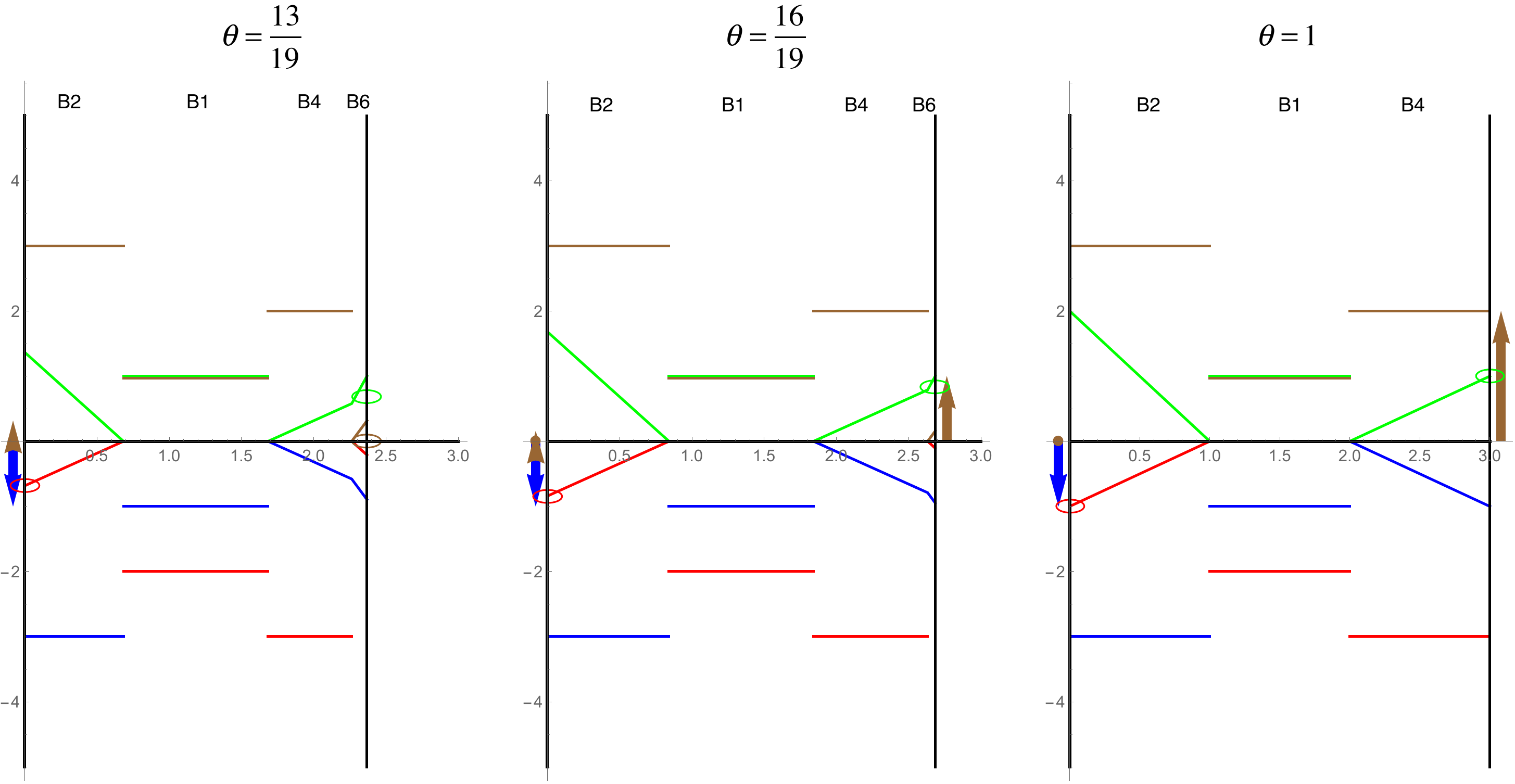}
\caption{Example problem, Iteration 7.}
\label{fig.p2iteration7}
\end{figure}
The result at $\theta=1$ is the optimal solution to our problem, as shown previously in Figure \ref{fig.p2full}.


\section{Structure of the Solution}
\label{sec.structure}
All the results in this section follow  directly from \cite{shindin-weiss:13,shindin-weiss:14}, and Appendix \ref{sec.lambdamu}.  We therefore state them here briefly, and the reader is advised that the notation and terms presented here will be used for the rest of the paper.

We  formulate a slightly  more general  M-CLP problem and its dual, which our algorithm will solve:
\begin{eqnarray}
\label{eqn.gmpclp}
&\max & \; \mu\Tt U(0) + \int_{0-}^T (\gamma+ (T-t)c)\Tt dU(t) \nonumber\\
\mbox{M-CLP}\quad & \mbox{s.t.} & \qquad A\, U(t) + x(t) \quad = \beta + b t, 
\quad 0 \le t < T, \\
&&  \qquad A\, U(T) + x(T) \quad = \beta + bT + \lambda, \nonumber \\
 &&   U(t)\ge0  \mbox{ non-decreasing  right continuous,}\; x(t)\ge 0,\;t\in [0,T],\nonumber 
\end{eqnarray}
\begin{eqnarray}
\label{eqn.gmdclp}
&\min & \; \lambda\Tt P(0) + \int_{0-}^T (\beta+(T-t)b)\Tt dP(t)  \nonumber \\
\mbox{M-CLP$^*$}\quad  & \mbox{s.t.} & \qquad  A\Tt P(t) - q(t) \quad = \gamma + c t,  
\quad 0 \le t  < T, \\
&& \qquad A\Tt P(T) - q(T) \quad = \gamma + cT + \mu, \nonumber\\
 &&   P(t)\ge 0 \mbox{ non-decreasing right continuous,}\; q(t)\ge 0,\;t\in [0,T].\nonumber 
\end{eqnarray}
The new elements in this re-formulation are  two additional vectors of boundary parameters: $\lambda\le 0,\,\mu\ge 0$.
The introduction of $\lambda,\,\mu$  here is done to avoid some inherent degeneracies that can inhibit our algorithm.  We discuss the motivation and explain the sign restrictions on $\lambda,\mu$, and we prove that all the results of  \cite{shindin-weiss:13,shindin-weiss:14}  extend directly to the current formulation, in Appendix \ref{sec.lambdamu}.

We  denote the complete set of {\em boundary parameters} by 
$\rho=(\beta, \gamma, \lambda,\mu,T)$.
We will often make the following non-degeneracy assumption:
\begin{assumption}[Non-Degeneracy]
\label{asm.nondeg}
The vector $b$ is in general position to the matrix $\left[ A\; I  \right]$ (i.e. it is not a linear combination of any less than $K$ columns), and the vector $c$ is in general position to the matrix $\left[ A\Tt \; I  \right]$.
\end{assumption}
Optimal solutions of (\ref{eqn.gmpclp}), (\ref{eqn.gmdclp}) are strongly dual and satisfy the following {\em complementary slackness condition}:
\begin{equation}
\label{eqn.compslack}
 \int_{0-}^T x(T-t)\Tt dP(t) =  \int_{0-}^T q(T-t)\Tt dU(t) = 0,
\end{equation}
Based on ideas of Wang, Zhang and Yao \cite{wang-zhang-yao:09}, and similar to Theorem S3.1, a necessary and sufficient condition that (\ref{eqn.gmpclp}) be feasible is that the following {\em Test-LP} is feasible
  \begin{eqnarray}
\label{eqn.ptestLP-ex}
 &\max &z = (\gamma + c T +\mu)\Tt   \bu + (\gamma+ cT)\Tt U  \nonumber  \hspace{1.0in} \\
    &\mbox{s.t.} &  A  \bu  \le \beta,   \\
  \mbox{Test-LP} &&  A  \bu + A U  \le \beta + bT + \lambda, \nonumber \\
&& \quad \bu,\, U \ge 0.  \nonumber 
\end{eqnarray}
An analogous condition applies for M-CLP$^*$ (\ref{eqn.gmdclp}).
Feasibility of (\ref{eqn.gmpclp}), (\ref{eqn.gmdclp}) implies existence of optimal complementary slack solutions with equal objective values.

Under the Non-Degeneracy Assumption \ref{asm.nondeg} optimal solutions $U(t),P(t)$ have impulse controls  $\bu^0=U(0),\,\bu^N=U(T)-U(T-),\,\bp^0=P(T)-P(T-),\,\bp^N=P(0)$ at $0$ and $T$,   piecewise constant control rates $u(t)=\frac{dU(t)}{dt},\,p(t)=\frac{dP(t)}{dt}$ at $0<t<T$, and continuous piecewise linear states $x(t)=\beta+bt-U(t),\,q(t)=\gamma+ct-P(t)$ at $0 \le t<T$, with possible discontinuities at $T$.  

The following list adds to the description of optimal solutions.
\begin{compactitem}[-]
\item
The time horizon $[0,T]$ is partitioned by $0=t_0 < t_1 < \ldots < t_N=T$ which are the breakpoints in the rates $u,p$  and in the slopes of $x,q$.  
\item
We denote the values of interval lengths by $\tau_n=t_n-t_{n-1},\,n=1,\ldots,N$.
\item
We denote the vectors of values of the states at the breakpoints by   $x^n = x(t_n),\, n=0,\ldots,N-1$,
and $q^n=q(T-t_n)$, $n=1,\ldots,N$.  
\item
Because there may be a discontinuity at $T$ we denote the values at $T$ itself by $\bx^N=x(T)$, $\bq^0=q(T)$, and let $x^N=x(T-)$, $q^0=q(T-)$ be the values of the limits as $t\nearrow T$. 
\item
The constant slopes of the states and the constant values of the control rates for each interval are denoted $\dx^n=\frac{d x(t)}{dt},\,u^n=u(t),\, t_{n-1}<t<t_n$ and  $\dq^n=\frac{d q(t)}{dt},\,p^n=p(t),\, T-t_n<t<T-t_{n-1}$.  
\end{compactitem}

The entire solution, $U(t),P(t),x(t),q(t)$ can be retrieved from  the vectors of boundary values $\bu^0,\bu^N,\bp^0,\bp^N$, $x^0,\bx^N,\bq^0,q^N$, the vectors of rates $u^n,\dx^n,p^n,\dq^n,$ $n=1,\ldots,N$, and and the interval lengths $\tau_n\,n=1,\ldots,N$.

We now continue to describe properties of optimal solutions.
The rates $u^n,\dx^n,p^n,\dq^n$ are complementary slack basic solutions of the following dual pairs of Rates-LP problems
\begin{equation}
\label{eqn.prates}
\begin{array}{rcl}
 & \max &  c\Tt u  \\
 &\mbox{s.t.}& A u +  \dx = b \\
\mbox{Rates-LP$(\K_n,\J_n)$} \quad&& u_j \in \setZ \text{ for } j \in \J_n,\; u_j \in \setP \text{ for } j \notin \J_n \\
&& \dx_k \in \setU \text{ for } k \in \K_n, \; \dx_k \in \setP \text{ for } k \notin \K_n 
\end{array}
\end{equation}
\begin{equation}
\label{eqn.drates}
\begin{array}{rcl}
 & \min &  b\Tt p  \\
&\mbox{s.t.}& A\Tt p -  \dq = c \\
\mbox{Rates-LP$^*(\K_n,\J_n)$}\quad && p_k \in \setZ \text{ for } k \in \K_n,\; p_k \in \setP \text{ for } k \notin \K_n \\
&& \dq_j \in \setU \text{ for } j \in \J_n, \; \dq_j \in \setP \text{ for } j \notin \J_n.
\end{array}
\end{equation}
Here $\setP,\setZ,\setU$ refer to sign restrictions on the variables:  $\setP$ means non-negative, $\setU$ means unrestricted, $\setZ$ means restricted to be 0.
For $n=1,\ldots,N$, $B_1,\ldots,B_N$ denote the optimal bases, and $\K_n,\J_n$ are the indexes of the basic $\dx_k^n,\dq_j^n$.  By Assumption \ref{asm.nondeg} all the bases are non-degenerate.  
We denote by $B^*_n$ the dual basis complementary to $B_n$.
\begin{definition}
We say that a base sequence $\{\K_n,\J_n\}_{n=0}^{N+1}$ is {\em a proper base sequence} if it satisfies:
\begin{compactitem}
\item
The bases $B_1,\ldots,B_N$ are {\em admissible} in the sense that  $u^n,p^n\ge 0$.
\item
For each $n = 1,\dots, N-1$ bases $B_n, B_{n+1}$ are  {\em adjacent},  in the sense that $B_n \to B_{n+1}$ involves a single pivot, with $v_n$  leaving the basis, and $w_n$ entering.
\item
The indexes of the non-zero boundary values of the primal and dual state variables  at times 0 and $T$ denoted by  $\{\K_n,\J_n\}_{n\in\{0,N+1\}}$ satisfy a {\em compatibility condition} that $\K_0 \subseteq \K_1,\,\J_{N+1} \subseteq \J_N$, i.e. if $x_k(0)>0$ then $\dx_k^1\ne 0$, and if $q_j(0)>0$ then $\dq_j^N\ne 0$.
\end{compactitem}
\end{definition} 

Given any proper base sequence, we can attempt to construct an optimal solution  by solving the Rates-LP/LP$^*$ (\ref{eqn.prates}), (\ref{eqn.drates}) for the bases $B_n, \,n=1,\ldots,N$ and then using these values to formulate and solve the following coupled linear equations, which determine the interval lengths and the boundary values.

{\em The time interval equations}: 
 \begin{eqnarray}
\label{eqn.lineq1}
&\displaystyle x_k(t_n)=x^0_k + \sum_{m=1}^n  \dx^m_k \tau_m = 0, &\mbox{ if } v_n = \dx_k, 
\nonumber \\
&\displaystyle q_j(T-t_n)=q^N_j + \sum_{m={n+1}}^N  \dq^m_j \tau_m = 0, &\mbox{ if } v_n = u_j,
\end{eqnarray}
\begin{equation}
\label{eqn.sumtau}
\sum_{n=1}^N \tau_n = T.
\end{equation}

{\em The complementary slackness conditions}:
\begin{equation}
\label{eqn.compatible}
\begin{array}{lllll}
\bu^0_j=0, &  j\in \J_0  &                                 \qquad \qquad \bp^N_k=0,  &  k\in \K_{N+1},  \\
x^0_k=0, &   k \not\in \K_0, &                          \qquad \qquad   q^N_j = 0, &  j\not\in \J_{N+1},  \\
\bp^0_k=0,  &  k\in \K_{0}, &  			 \qquad \qquad \bu^N_j=0, & j\in \J_{N+1}, \\
\bq^0_j = 0, & j\not\in \J_0, &			 \qquad \qquad \bx^N_k= 0, & k\not\in \K_{N+1}.
\end{array}
\end{equation}

{\em The first boundary equations}:
 \begin{equation}
\label{eqn.boundary1}
\begin{array}{llll}
A \bu^0  &+ x^0 = \beta, \quad  & \quad   A\Tt \bp^N  &- q^N = \gamma. 
\end{array}
\end{equation}

{\em The second boundary equations}:
\begin{equation}
\label{eqn.boundary2}
A \bu^N  +  \bx^N - x^N = \lambda, \qquad  \qquad   A\Tt \bp^0   -  \bq^0 + q^0= \mu.
\end{equation}
Where all the values of the state variables are obtained from:
\begin{equation}
\label{eqn.lineq2}
 \begin{array}{ll}
\displaystyle x_k^n=x^0_k + \sum_{m=1}^n  \dx^m_k \tau_m, &  k=1,\ldots,K,\;n=1,\ldots,N,  \\
\displaystyle q_j^n=q^N_j + \sum_{m={n+1}}^N  \dq^m_j \tau_m, & 
j=1,\ldots,J,\;n=1,\ldots,N.
\end{array}
\end{equation}

Theorem  S3.1 in \cite{shindin-weiss:14} states that for a proper base sequence, if the solution to (\ref{eqn.lineq1})--(\ref{eqn.lineq2})  is non-negative then it is an optimal solution, and every feasible M-CLP/M-CLP$^*$ problem has an optimal solution of this form, i. e. one can construct a proper base sequence such that solution of (\ref{eqn.lineq1})--(\ref{eqn.lineq2}) will be non-negative.

Given a proper base sequence, after solving (\ref{eqn.prates}), (\ref{eqn.drates}),  we can formulate all the equations (\ref{eqn.lineq1})--(\ref{eqn.lineq2}) as 
\begin{equation}
\label{eqn.ex-coupled}
M \bH = \bR.  \qquad 
\end{equation}
The matrix $M$ is a square matrix  of dimension $(J+K)(N+4)+N$, defined in Section S3.5
\footnote{There are some typos in the definition of $M$ in \cite{shindin-weiss:14}.  They were corrected in \cite{shindin:16}.}.  The vector of unknowns (writing all vectors as row vectors) is:
\[
 \begin{array}{ccccccccccccc}
\bH = [ & \bu^0 & x^0 &  \bp^N & q^N & \tau & x^1 \dots x^N  & q^0 \dots q^{N-1} & \bu^N & \bx^N &\bp^0 & \bq^0 & ]\Tt
\end{array}
\]
and the right hand side consists of the boundary parameters $\rho$ interspersed by properly dimensioned 0 vectors:
\begin{equation}
\label{eqn.transform}
\rho \to \bR: \bR = \left[ \begin{array}{cccccccccccccccc}
\beta\Tt & 0  & 0 & \gamma\Tt & 0 & 0 & 0 & T & 0 & 0 & \lambda\Tt & 0 & 0 & \mu\Tt & 0 & 0
\end{array}\right]\Tt
\end{equation} 

The set of equations (\ref{eqn.ex-coupled}) are formulated in such a way that they force $2K+ 2J$ boundary variables, $N-1$ variables that are left-hand sides of (\ref{eqn.lineq1}) and $N(K+J)- \sum_{n=1}^N (|\K_n|+|\J_n|)$ of the $x_k^n \in \{n,k:\dx^n_k = 0\},\, q_j^n \in \{n,j: \dq^{n+1}_j = 0\}$ that correspond to non-basic variables in the Rates-LP, to be zero. 
We denote this set of zero-valued variables by $\bH_\setZ$.  The set of remaining,  non-negative variables, is denoted by $\bH_\setP$. 
We say that the solution $H$ is {\em fully non-degenerate} if all components of  $\bH_\setP$ are $>0$.


\section{Decomposition}
\label{sec.decomp}

We introduce the following distinction between the state variables that are indexed by $\K_0,\J_{N+1}$:
\begin{definition}
Consider a fully non-degenerate solution, and the $x_k(t)$ with $k\in \K_0$.  If for some $t_n < T$, $x_k(t_n) = 0$ then we say that $x_k$ is tied.  On the other hand, if $x_k(t)>0$ for $0\le t < T$ then we say that $x_k$ is free.  We denote by $\K^=$ the indexes of the tied variables, and by $\K^{\uparrow\downarrow}$ the indexes of the free variables, so that $\K_0 = \K^= \cup \K^{\uparrow\downarrow}$.  We define $\J^=,\, \J^{\uparrow\downarrow}$ analogously for $q_j(t), j\in \J_{N+1}$.
\end{definition}
The tied state variables appear both in the boundary equations  (\ref{eqn.boundary1}) and in the time interval equations (\ref{eqn.lineq2}), and introduce coupling between these different sets of equations.  The free state variables appear only in the boundary equations (\ref{eqn.boundary1}).

Given an optimal solution of an M-CLP/M-CLP$^*$ problem, with data $A,b,c,\beta,\gamma,\lambda,\mu,T$, we introduce some more notations for various quantities:
\begin{eqnarray}
\label{eqn.xlevels}
&&  x_k^\bullet = \min_{0\le t < T} x_k(t),  \quad  \tilde{x}_k = x_k^0 -  x_k^\bullet, \quad
 k=1,\ldots,K,  \\
&&  q_j^\bullet  = \min_{0 < t \le T} q_j(t),  \quad  \tilde{q}_j = q_j^N -  q_j^\bullet,
 \quad j=1,\ldots,J.
 \end{eqnarray} 
Here $x_k^\bullet$ is the minimum of $x_k(t)$ over the range $0 \le t <T$ (excluding the value $x_k(T)=\bx^N$).  The values of  $ \tilde{x}_k$ can also be expressed as $\tx_k = - \min_{0 \le n \le N} \sum_{m=1}^n \dx_k^m \tau_m$ (where empty sum equals 0).  
 For $k\in \K^=$, $x_k^\bullet= 0$ and $\tilde{x}_k= x_k^0$.  For $k\in \K^{\uparrow\downarrow}$, $x_k^\bullet > 0$,  and $\tilde{x}_k = x_k^0 - x_k^\bullet\ge 0$.  For $k\not\in \K_0$, we have $x_k^0=x_k^\bullet=\tilde{x}_k=0$.  Analogous statements hold for $q_j(t)$.

We also introduce the following notation for the vectors of total cumulative controls excluding the impulse controls:
\begin{equation}
\label{eqn.cumcontrols}
 \tU=\int_0^T u(t) dt, \qquad \tP= \int_0^T p(t) dt,
\end{equation}
We refer to $\tU,\tP, \tx, \tq$ as decomposition parameters.
\begin{proposition}
\label{thm.uniqueparams}
Under Non-Degeneracy Assumption \ref{asm.nondeg} M-CLP/M-CLP$^*$ possess a unique set of decomposition parameters $\tU,\tP, \tx, \tq$.  They are all of them affine functions of elements of $\rho$.
\end{proposition}
\begin{proof}
Uniqueness  follows directly from Theorem D5.5(iii), which states that under Assumption 
\ref{asm.nondeg} $u(t),p(t)$ are uniquely determined, and as a result so are  $\dx(t),\dq(t)$.  Proof that they are affine functions of $\rho$ follows similar to the proof of  Corollary S3.8.  
%
%
 \rule{0.5em}{0.5em} \end{proof}

We are now ready to define a decomposition of M-CLP/M-CLP$^*$ to two parametric families of problems.  The first consists of  SCLP/SCLP$^*$ problems and concerns the solution in the interior of the time horizon.  The second  consists of pairs of Boundary-LP problems, and concerns the solution on the boundaries of the time horizon.

\subsubsection*{Internal-SCLP Problems:}
Consider M-CLP/M-CLP$^*$ problems with data $A,b,c$. 
We define a parametric family SCLP/ \\ SCLP$^*(x^0,q^N)$ with non-negative parameters $x^0, q^N$
\begin{eqnarray}
\label{eqn.PWSCLP4}
&\max & \int_0^T (-q^{N} + (T-t)c)\Tt u(t) \,dt   \nonumber   \hspace{0.9in} \\
\mbox{SCLP$(x^0,q^N)$} \quad & \mbox{s.t.} &  \int_0^t A\, u(s)\,ds  \le x^0 + b t, \\
&& \quad u(t)\ge 0, \quad 0 \le t \le T.  \nonumber
\end{eqnarray}
\begin{eqnarray}
\label{eqn.DWSCLP4}
&\min & \int_0^T ( x^0 + (T-t)b)\Tt p(t) \,dt     \nonumber  \hspace{0.9in} \\
\mbox{SCLP$^*(x^0,q^N)$}\quad  &\mbox{s.t.} &  \int_0^t  A\Tt\, p(s)\,ds \ge - q^{N}+ c t, \\
&& \quad p(t)\ge 0, \quad 0 \le t \le T.   \nonumber
\end{eqnarray}
These problem are a special case of separated continuous linear program (SCLP) of the form discussed in \cite{weiss:08}. We call these problems {\em  Internal-SCLP/ SCLP$^*$}.

\subsubsection*{Boundary Problems}
Consider an M-CLP/M-CLP$^*$ problem with  data $A,b,c,\beta,\gamma,T,\lambda,\mu$.  We define, similar to S15, S16 in \cite{shindin-weiss:14}, a family of pairs of LP problems,  for parameter vectors $\tilde{u},\tilde{x}, \tU, \tP$:
 \begin{equation}
\label{eqn.modprimalpbLP}
\mbox{\rm Boundary-LP$(\tU,\tilde{x})$ }\quad 
\begin{array}{ll}
\max & (\gamma +c T + \mu)\Tt \bu^0 + \gamma\Tt  \bu^N   \\
\mbox{s.t.}& A \bu^0  + x^\bullet     \qquad \quad           = \beta - \tilde{x},  \\
                 & A \bu^0  + A \bu^N  + \bx^N = \beta +  bT - A \tilde{U} + \lambda,  \\
& \bu^0,\,\bu^N,\,  x^\bullet,\,\bx^N \ge 0. \\
\end{array}  
\end{equation}
\begin{equation}
\label{eqn.moddualpbLP}
\mbox{\rm Boundary-LP$^*(\tilde{P},\tilde{q})$}\quad 
\begin{array}{ll}
\min & (\beta +b T + \lambda)\Tt \bp^N + \beta\Tt \bp^0   \\
\mbox{s.t.}&  A\Tt \bp^N   - q^\bullet    \qquad \quad           = \gamma + \tilde{q},  \\
                 & A\Tt \bp^N  + A\Tt \bp^0  - \bq^0 = \gamma +  c T - A\Tt \tilde{P} + \mu, \\
& \bp^N,\,\bp^0,\, q^\bullet,\,\bq^0 \ge 0.
\end{array}  
\end{equation}
Note, that Boundary-LP$(\tU,\tilde{x})$ and Boundary-LP$^*(\tilde{P},\tilde{q})$ share the same matrix of constraint coefficients (transposed), but they are not dual to each other.  Note also that they use disjoint sets of decomposition parameters ($\tU, \tx$ for Boundary-LP and $\tilde{P},\tilde{q}$ for  Boundary-LP$^*$).

\begin{theorem}
\label{thm.moddecomp}
(i) An optimal solution of M-CLP/M-CLP$^*$ can be decomposed to  optimal solutions of  Internal-SCLP/SCLP$^*(x^0, q^N)$ and feasible solutions of Boundary-LP$(\tU, \tx)$ and Boundary-LP$^*(\tP, \tq)$  satisfying: 
\begin{equation}
\label{eqn.boundarycomp}
\begin{array}{c}
\displaystyle 
x^0=x^\bullet + \tilde{x}, \quad q^N = q^\bullet+\tilde{q}, \quad  
 \tU=\int_0^T u(t) dt, \quad \tP= \int_0^T p(t) dt, \\
\displaystyle \tilde{x}_k = - \min_{1\le n \le N} \left(0, \sum_{m=1}^n \dx_k^m \tau_m \right),
\quad
\tilde{q}_j = - \min_{1\le n \le N} \left(0, \sum_{m=n}^N \dq_j^m \tau_m \right), \\
\displaystyle {x^0}\Tt \bp^0 = {\bx^N}\Tt \bp^N = {q^N}\Tt \bu^N = {\bq^0}\Tt \bu^0 = 0.
 \end{array}  
\end{equation}

(ii) Conversely, a combination of optimal solutions of  Internal-SCLP/ SCLP$^*(x^0, q^N)$, and of feasible solutions of Boundary-LP$(\tU, \tx)$ and Boundary-LP$^*(\tP, \tq)$, that satisfy (\ref{eqn.boundarycomp}) can be composed into an optimal solution of M-CLP/M-CLP$^*$.

(iii) The solutions of Boundary-LP$(\tU, \tx)$ and Boundary-LP$^*(\tP, \tq)$ considered in (i) and (ii) are optimal.
\end{theorem}
\begin{proof}  
(i) Let $u(t), x(t), p(t), q(t), \bu^0, \bu^N, x^0, \bx^N, \bp^N, \bp^0, q^N, \bq^0$ be an optimal solution of M-CLP/M-CLP$^*$ as described in Theorem S3.1. Then it is immediate to see that the same $u(t), x(t), p(t), q(t)$ solve the Internal-SCLP/SCLP$^*$ $(x^0, q^N)$. 

Furthermore, one can see that $\bu^0, \bu^N, x^\bullet = x^0 - \tx, \bx^N$ is a feasible solution of Boundary-LP$(\tU, \tx)$ and $\bp^N, \bp^0, q^\bullet = q^N -\tq, \bq^0$ is a feasible solution of Boundary-LP$^*(\tP, \tq)$.

(ii) On the other hand, optimal solution of SCLP/SCLP$^*$ provides $u(t),$ $p(t),x(t),q(t)$ and solutions of Boundary-LP/LP$^*$ provides $\bu^0,\bu^N,x^\bullet, \bx^N,\bp^0,$ $\bp^N,q^\bullet, \bq^0$. It is immediate to see that if these satisfy (\ref{eqn.boundarycomp}) then they  provide a feasible complementary slack solution of  M-CLP/M-CLP$^*$. 

(iii) This follows from Corollary S3.3, where a proof was given for  $\lambda = \mu = 0$. The result for  $\lambda < 0, \mu > 0$ is essentially same. \qquad
 \rule{0.5em}{0.5em} \end{proof}

\begin{corollary}
\label{thm.forcedzero}
Consider a solution of Internal-SCLP/SCLP$^*$ and of Boundary-LP/LP$^*$ that satisfies (\ref{eqn.boundarycomp}).   Then for the solution of Boundary-LP/LP$^*$  the following holds:

(i) The solutions of Boundary-LP$(\tilde{U}, \tilde{x})$ and of Boundary-LP$^*(\tilde{P}, \tilde{q})$ are complementary slack.

(ii) For $k\in \K^=$ we have $x_k^\bullet=0$ , and for $j\in \J^=$ we have $q_j^\bullet=0$.

(iii) For $j\in \J^=$ we have $\bu_j^N = 0$, and for $k\in \K^=$ we have $\bp_k^0=0$. 
\end{corollary}
\begin{proof}
One can see that from (\ref{eqn.boundarycomp}) follows:
\[
{x^\bullet}\Tt \bp^0 = {\tx}\Tt \bp^0 = {\bx^N}\Tt \bp^N = {q^\bullet}\Tt \bu^N = {\tq}\Tt \bu^N = {\bq^0}\Tt \bu^0 = 0.
\]
 In particular this implies:
 
(i) Variables $\bu_j^0,\bq_j^0$, variables $\bp_k^N,\bx_k^N$, variables $x_k^\bullet,\bp_k^0$ and variables $q_j^\bullet,\bu_j^N$ are complementary slack.

(ii)  In the solution of  Internal-SCLP/SCLP$^*(x^0,q^N)$, $x_k^0 = \tilde{x}_k,\,k\in\K^=$, and hence $x_k^\bullet=0$ for $k\in\K^=$.  Similarly, $q_j^\bullet=0$ for $j\in \J^=$.

(iii)  By complementary slackness of $x^0,\bp^0$ and $q^N,\bu^N$, we also have that $\bp_k^0=0$ for $k\in \K^=$, $\bu_j^N = 0$ for $j\in \J^=$.
 \rule{0.5em}{0.5em} \end{proof}
The problems Boundary-LP($\tU, \tx$), Boundary-LP$^*(\tP, \tq)$ may have non-unique optimal solutions for a given set of parameters $\tU, \tx, \tP, \tq$.  If there are multiple solutions then it is not obvious how to check if (\ref{eqn.boundarycomp}) holds for some solutions.  The following Theorem answers these questions.
\begin{theorem}
\label{thm.corresp}
Under the Non-Degeneracy Assumption \ref{asm.nondeg} there is a one to one correspondence between optimal solutions of M-CLP and Boundary-LP($\tU, \tx$). The same holds for optimal solutions of M-CLP$^*$ and Boundary-LP$^*(\tP, \tq)$.
\end{theorem}
\begin{proof} From Proposition \ref{thm.uniqueparams} it follows that under Non-Degeneracy Assumption \ref{asm.nondeg} the formulation of Bound\-ary-LP($\tU, \tx$) problem is unique, whether M-CLP has a unique solution or not.

Furthermore, each optimal solution of M-CLP determines the boundary values $\bu^0,\bu^N,\bx^N$ and the values of $x^\bullet = x^0 - \tx$ that by Theorem \ref{thm.moddecomp}(iii) is an optimal solution of Boundary-LP($\tU, \tx$). Hence, each optimal solution of M-CLP uniquely defines an optimal solution of the Boundary-LP($\tU, \tx$).

Conversely, each optimal solution of Boundary-LP($\tU, \tx$) together with $\dx(t), u(t)$ that are unique by Theorem D5.5(iii) produces a feasible solution of M-CLP. Moreover, objective values of these solutions satisfy:
\[ V(\text{M-CLP}) = V(\text{Boundary-LP}(\tU, \tx)) + \int_{0+}^{T-} (\gamma + c(T-t))\Tt u(t) dt. 
\]
The second term on the right hand side is uniquely determined, and therefore it follows that all the optimal solutions of Boundary-LP($\tU, \tx$) produces a feasible solution of M-CLP with same objective value.   
 However, by Theorem \ref{thm.moddecomp} one of the optimal solutions of Boundary-LP($\tU, \tx$),  is obtained by the decomposition of an optimal solution of M-CLP and hence each optimal solution of  Boundary-LP($\tU, \tx$) uniquely defines an optimal solution of M-CLP.
 \rule{0.5em}{0.5em} \end{proof}

{\bf Note:}  Theorem \ref{thm.corresp} explains why we choose the formulations  (\ref{eqn.modprimalpbLP}), (\ref{eqn.moddualpbLP})  of the Boundary-LP/LP$^*$ problems.  This one to one correspondence of optimal solutions that is shown in the Theorem does not exist for the  more natural pair of boundary problems (S13), (S14) in \cite{shindin-weiss:14}.  

\subsection{Boundary Simplex Dictionary}
\label{sec.dictionary}

By the complementary slackness condition in equations (\ref{eqn.compatible}), there are 
 exactly $2K+2J$ boundary values which are 0, and up to $2K+2J$ positive ones.  
  However, as it turns out, the number  of positive primal boundary values may be different from $2K$.  In fact the number of positive primal boundary values is $|\K_0|+|\overline{\J_0}|+|\K_{N+1}|+|\overline{\J_{N+1}}| = 2K+L$, and correspondingly,   
 the number of positive dual boundary values $|\J_{N+1}|+|\overline{\K_{N+1}}|+|\J_0|+|\overline{\K_0}| = 2J-L$ where $L$ may be either 0 or positive or negative.   In other words, the set of columns of $\left[\begin{array}{cccc} 
 A & 0 & I & 0 \\ A & A & 0 & I  \end{array}\right]$ that correspond to positive primal boundary values may not be a basis,  and similarly for the dual.

Recall also that the problems Boundary-LP$(\tU, \tx)$ and Boundary-LP$^*(\tP, \tq)$ are not dual to each other.  However, optimal solutions of Boundary-LP$(\tU, \tx)$ and Boundary-LP$^*(\tP, \tq)$ that correspond to optimal solutions of M-CLP/M-CLP$^*$ are unique and are  complementary  slack.  If possible we wish to 
find basic solutions that are summarized by a single   {\em  boundary simplex dictionary}.  This dictionary will enables us to define a boundary pivot operation in Section \ref{sec.bound_pivot}.   The following theorem states when this is possible.
\begin{theorem}
\label{thm.basis}
If M-CLP/M-CLP$^*$ has a unique solution with $H_\setP>0$, then its optimal base sequence  satisfies $|\K^{\uparrow\downarrow}| + |\overline{\J_0}| + |\K_{N+1}| + |\overline{\J_{N+1}}| \le 2K$ and $|\J_0| + |\overline{\K_0}| + |\J^{\uparrow\downarrow}| + |\overline{\K_{N+1}}| \le 2J$.
\end{theorem}
{\bf Note}:  while $|\K_0|+|\overline{\J_0}|+|\K_{N+1}|+|\overline{\J_{N+1}}|$ may be larger than $2K$, the theorem states that under the given conditions, if $|\K_0|$ is replaced by $|\K^{\uparrow\downarrow}|$, the sum cannot exceed $2K$.
\begin{proof}
If M-CLP/M-CLP$^*$ has a unique solution with $H_\setP > 0$,
then the optimal solution of the corresponding  Boundary-LP (\ref{eqn.modprimalpbLP}) has $|\K^{\uparrow\downarrow}| + |\overline{\J_0}| + |\K_{N+1}| + |\overline{\J_{N+1}}|$ positive variables.
But if $|\K^{\uparrow\downarrow}| + |\overline{\J_0}| + |\K_{N+1}| + |\overline{\J_{N+1}}| >2 K$ this implies that it has a non-basic optimal solution, and therefore its solution is non-unique.  But then by Theorem \ref{thm.corresp} M-CLP has non-unique optimal solutions, which is a contradiction.  
 \rule{0.5em}{0.5em} \end{proof}

We now  decompose the matrix of coefficients of (\ref{eqn.modprimalpbLP}) $\left[ \begin{array}{cccc} A & 0 & I & 0 \\ A & A & 0 & I \end{array} \right]$ into a basic and non-basic part, with similar decomposition for the matrix of coefficients of (\ref{eqn.moddualpbLP})  $\left[ \begin{array}{cccc} A\Tt & 0 & -I & 0 \\ A\Tt & A\Tt & 0 & -I \end{array} \right]$.
Let $\bB$ be a basic submatrix of the former, $\bB^*$ a basic matrix of the latter, such that they are complementary slack, and let $\bN$, $\bN^*$  be the matrices composed from the corresponding non-basic columns, so that (as seen for example in \cite{vanderbei:14}) $\bB^{-1} \bN = (\bB^{*-1} \bN^*)\Tt$.  
We then say that $\bB$ and $\bB^*$ are {\em compatible}.

We now consider a base sequence that satisfies $|\K^{\uparrow\downarrow}| + |\overline{\J_0}| + |\K_{N+1}| + |\overline{\J_{N+1}}| \le 2K$ and $ |\J^{\uparrow\downarrow}| + |\overline{\K_{N+1}}|+|\J_0| + |\overline{\K_0}| \le 2J$.  For such a base sequence we now specify rules to construct  a pair of compatible bases of boundary variables:
\begin{compactitem}[-]
\item define primal variables with indexes that are members of $\K^{\uparrow\downarrow}, \overline{\J_0}, \K_{N+1},$ $ \overline{\J_{N+1}}$ as basic variables for $\bB$ and dual variables with indexes that are members of $\J^{\uparrow\downarrow}, \overline{\K_{N+1}}, \J_{0}, \overline{\K_{0}}$ as basic variables for $\bB^*$,
\item if the solution of M-CLP possesses $\ge 2K$ strictly positive primal boundary values then choose an arbitrary subset of size $|\K_0| + |\overline{\J_0}| + |\K_{N+1}| + |\overline{\J_{N+1}}| - 2K$ of indexes out of $\K^=$ and define those $\bp_k^0$ as dual basic variable in $\bB^*$, and for the remaining indexes of $\K^=$ define $x_k^\bullet$ as primal basic variable in $\bB$.  Recall that for all the variables in $\K^=$, $x_k^\bullet = 0$, and at the same time also $\bp_k^0 = 0$ by complementary slackness.  So there is no problem in naming either  $x_k^\bullet$ as primal basic, or $\bp_k^0$ as dual basic, for an arbitrary $k\in \K^=$.

Define all the variables $q_j^\bullet,\,j\in \J^=$ as dual basic variables in $\bB^*$. 
\item similarly,
 if the solution of M-CLP$^*$ possesses $\ge 2J$ strictly positive boundary values,
 choose an arbitrary subset of size $|\J_{N+1}| + |\overline{\K_{N+1}}| + |\J_0| + |\overline{\K_0}| - 2J$ of indexes out of $\J^=$ and define those $\bu_j^{N}$ as  basic variables in $\bB$, and for the remaining indexes of $\J^=$ define $q_j^\bullet$ as dual basic variable in $\bB^*$.  Define all the variables $x_k^\bullet,\,j\in \K^=$ as primal basic variables in $\bB$. 
%
\end{compactitem} 
Using compatible boundary bases we define a boundary simplex dictionary $\D$ by:
\begin{equation}
\label{eqn.boundarybasis}
\D = \begin{array}{c|c}
&  \bv^*_{\B^*} \\
\hline \\
\bv_{\B} & \hbA = \bB^{-1} \bN = (\bB^{*-1} \bN^*)\Tt
\end{array}
\end{equation}
where $\bv = \{x^\bullet = x^0 - \tx, \bu^0, \bx^N, \bu^N \}$, $\bv^* = \{q^\bullet =  q^N - \tq, \bp^N, \bq^0, \bp^0 \}$ are values of variables of Boundary-LP/LP$^*$ and $\B, \B^*$ are sets of indexes of  the corresponding basic variables. 

In our algorithm, at iterations in which a boundary pivot will be required, changes in $(\K_n,\J_n)_{n\in\{0,N+1\}}$ will be achieved by pivots of this specially constructed dictionary.  In the proof of Theorem \ref{thm.boundary-unique}, in Appendix \ref{sec.uniqueproof}
we show that the arbitrary choice involved in constructing the compatible bases for the dictionary does not lead to any ambiguity.


\section{Validity Regions}
\label{sec.region}
Throughout this section we assume that Non-Degeneracy Assumption \ref{asm.nondeg} holds. 
Similar to the definition of validity region S3.6 we define:
\begin{definition}
Let $(\K_n, \J_n)_{n =0}^{N+1}$ be a proper base sequence. Let $\V$  be the set of all $\rho = (\beta, \gamma,  \lambda, \mu, T)$ for which this base sequence is optimal. Then  $\V$ is called the validity region of $(\K_n, \J_n)_{n =0}^{N+1}$.  Let $\F$ be the union of all validity regions, then $\F$ is called the parametric feasible region.
\end{definition}
As in Theorem S3.7, each validity region is a closed convex polyhedral cone (it may consist only of the origin). The parametric feasibility region is clearly convex, and it is a union of validity regions, hence $\F$  is also a closed convex polyhedral cone.
 
Similar to Corollary S3.8 we can state:
\begin{corollary}
\label{thm.affine}
For a given optimal base sequence, within its validity region, the elements of $\bH_\setP$ are  affine  functions of the non-zero elements of $\bR$.  Conversely, consider all optimal solutions  $\bH$ that belong to a given optimal base sequence.  Then the non-zero elements of $\bR$ change linearly as a function of non-zero elements of $\bH_\setP$.
\end{corollary}

\begin{proposition}
\label{thm.number}
The number of validity regions is bounded by $\binom{4(K+J)}{2(K+J)}\; 2^{\binom{K+J}{K}}$.
\end{proposition}    
\begin{proof}
The total number of bases of the Rates-LP/LP$^*$ is $ \binom{K+J}{K}$. Moreover, for a proper base sequence, under the non-degeneracy assumption, the  objective values of the Rates-LP/LP$^*$ are strictly decreasing in $n$ (see Corollary 4 in \cite{weiss:08}) and hence the number of sequences $(\K_n, \J_n)_{n=1}^{N}$ is bounded by  $2^{\binom{K+J}{K}}$. The number of $(\K_n, \J_n)_{n \in \{0, N+1\}}$ is bounded by $\binom{4(K+J)}{2(K+J)}$ and hence the total number of proper base sequences is bounded by $\binom{4(K+J)}{2(K+J)}\; 2^{\binom{K+J}{K}}$.
 \rule{0.5em}{0.5em} \end{proof}
{\bf Note}, a simpler somewhat larger bound is  $2^{2^{K+J+1}}$.

\begin{theorem}
\label{thm.interior}
For a proper base sequence $(\K_n, \J_n)_{n=0}^{N+1}$ with validity region $\V$, and for boundary parameters $\rho$ in the interior of $\F$,  the following statements are equivalent: 

(i) $\rho$ is an interior point of $\V$,

(ii) $\rho$ does not belong to the validity region of any other proper base sequence,

(iii) the solution of (\ref{eqn.ex-coupled}) for the point $\rho$ is unique and satisfies $\bH_\setP>0$.
\end{theorem}
\begin{proof}  See Appendix \ref{sec.validityappendix}  \rule{0.5em}{0.5em} \end{proof}
Note that Theorem \ref{thm.interior} implies that interiors of validity regions are disjoint.

\begin{corollary}
\label{thm.nonsingular}
If a proper base sequence has a  validity region with non-empty interior, then the corresponding  matrix $M$ is non-singular.
\end{corollary}
\begin{proof}
By Theorem \ref{thm.interior}, at an interior point of the validity region the solution of (\ref{eqn.ex-coupled}) is unique, and hence $M$ is non-singular. 
 \rule{0.5em}{0.5em} \end{proof}

\begin{corollary}
\label{thm.interior-unique}
If a point $\rho$ is an interior point of some validity region then the solutions of M-CLP and M-CLP$^*$ for this point are unique.
\end{corollary}
\begin{proof}
By Theorem \ref{thm.interior}  the sequence of bases is unique, $u(t),p(t),\dx(t),\dq(t)$ are unique by Theorem D5.5(iii), and so $M$ is unique, and the solution is unique.
 \rule{0.5em}{0.5em} \end{proof}


 \section{Collisions}
 \label{sec.collisions}
In this section we describe how the solution changes as the boundary parameters move from an internal point of a validity region to a boundary point of this validity region.  We consider a validity region $\V$ with non-empty interior
and base sequence $(\K_n, \J_n)_{n=0}^{N+1}$.  Let $\L(\theta)$ be a parametric line of boundary parameters, with $\L(\theta)$ in the interior of $\V$ for $\theta < \otheta$, $\L(\otheta)$ on the boundary, with $H_\setP(\theta)$ the solution at  $\L(\theta)$.  Then $H_\setP(\theta)>0$ for $\theta < \otheta$, but some elements of $H_\setP$ shrink to 0 as $\theta \nearrow \otheta$.  We call this {\em a collision}.

We now list all types of collisions.  
For convenience we number them  (a), (b), $\ldots$, (f) .

\paragraph{Internal collisions}
\begin{compactdesc}
\item[(a)]{\bf  State collision.} A single value $x^n_k, n=1,\ldots,N$, or $q^n_j, n=0,\ldots,N-1$, shrinks to $0$ (see Figure \ref{fig.cola}). This is classified in \cite{weiss:08} as SCLP-III type collision.
\begin{figure}[h!t] 
\centering 
\includegraphics[width=4.6in]{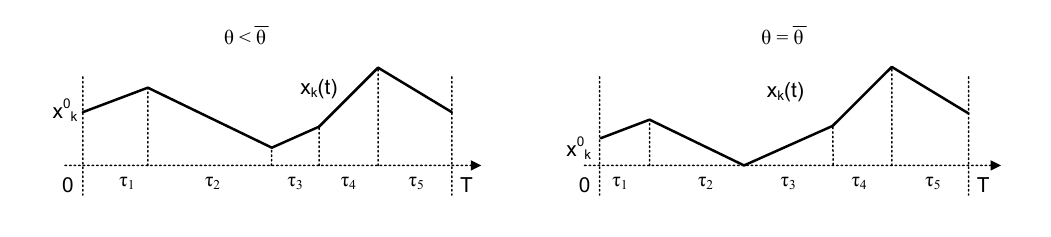} 
\caption[Type (a) collision, $x^2_k$ shrinks to $0$.]{Type (a) collision, $x^2_k$ shrinks to $0$.} 
\label{fig.cola} 
\end{figure}
\item[(b)]{\bf Interval collision between non-adjacent bases.} A sequence of one or more internal interval lengths, $\tau_{n_{*}},\ldots,\tau_{n_{**}}$ shrink to 0, where $1 < n_* \le n_{**} < N$ and 
$|B_{n_*-1} \setminus B_{n_{**}+1} | =2$, i.e. the basis preceding $n_*$ and the 
basis succeeding $n_{**}$ are not adjacent.  This means that at the collision point on the boundary, at time $t = t_{n_*}= \cdots = t_{n_{**}}$  two different state variables hit the value 0 ($x_k(t)$ and $x_{k'}(t)$, or $q_j(T-t)$ and $q_{j'}(T-t)$ or  $x_k(t)$ and $q_j(T-t)$ (see Figure \ref{fig.colb})).  
This is  classified in \cite{weiss:08} as SCLP-II type collision.
\begin{figure}[H] 
\centering 
\includegraphics[width=4.6in]{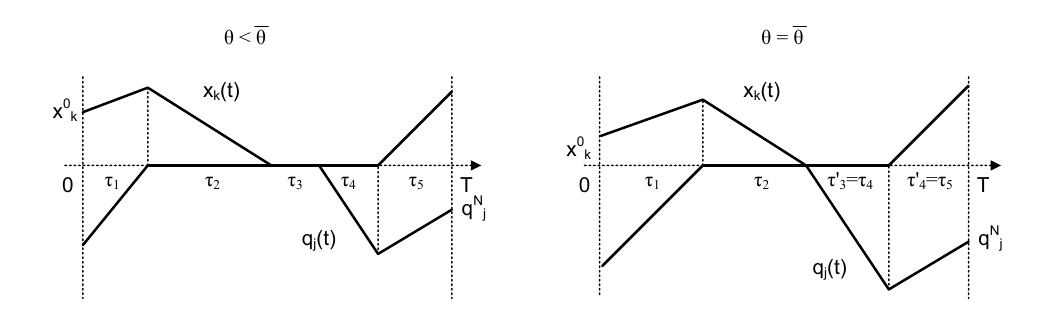} 
\caption[Type (b) collision, $\tau_3$ shrinks to $0$.]{Type (b) collision, $\tau_3$ shrinks to $0$.} 
\label{fig.colb} 
\end{figure}

\item[(c)]{\bf Interval collision between adjacent bases or at 0 or $T$,  
variable does not become free.} 
Includes the following three cases:  
(i) A sequence of one or more  interval lengths, $\tau_{n_{*}},\ldots,\tau_{n_{**}}$ shrink to 0, where $1 < n_* \le n_{**} < N$ and 
$|B_{n_*-1} \setminus B_{n_{**}+1} | =1$, (ii) intervals $\tau_1,\ldots,\tau_{n_{**}},$ $1\le n_{**}<N$ shrink to zero, (iii) intervals $\tau_{n_*},\ldots,\tau_N,\,1<n_*\le N$ shrink to zero.  In  case (i), a state variable $x_k$ (or a dual state variable  $q_j$) is positive in the intervals $t_{n_*-1} < t < t_{n_*}$ and $t_{n_{**}} <t < t_{n_{**}+1}$, but it is 0 at $t_{n_*}$ and at $t_{n_{**}}$.  In the  case (ii) the collision is at $t= 0$ and a dual state variable $q_j$ in positive in  $t_{n_{**}} <t < t_{n_{**}+1}$ but it is 0 at $t_{n_{**}}$.  In  case (iii) the collision is at  $t=T$, and a  state variable $x_k$ in positive in  $t_{n_*-1} <t < t_{n_*}$ but it is 0 at $t_{n_*}$.   
  These collisions are classified as SCLP-I type collisions in \cite{weiss:08}.
\begin{figure}[H] 
\centering 
\includegraphics[width=4.6in]{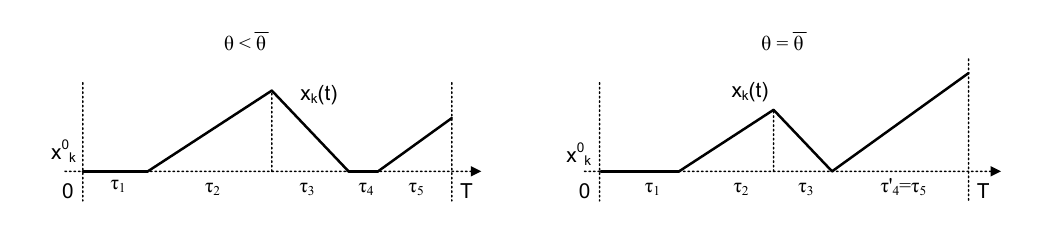} 
\caption[Type (c) collision, $\tau_4$ shrinks to $0$, $x_k$ remains tied.]{Type (c) collision, $\tau_4$ shrinks to $0$, $x_k$ remains tied ($n_*=n_{**}= 4$).} 
\label{fig.colc} 
\end{figure}

In addition, we also assume that in validity region $\V$, for the state variable $x_k(t)$ (or $q_j(t)$) there is 
at least one  time interval in $[0,t_{n_*}]$ or in  $[t_{n_{**}}, T]$  such that $x_k(t)=0$ (or $q_j(T-t)=0$) in that interval.   This implies that the variable $x_k(t)$ (or $q_j(t)$) would be tied even if $x_k(t)$  (or $q_j(T-t)$) would be positive in the whole interval $[t_{n_*},t_{n_{**}}]$.  This additional condition is what is meant by saying the variable involved in the collision does not become free.  In Figure \ref{fig.colc} a single interval  $\tau_{n^*}=\tau_{n^{**}} = \tau_4$ shrinks to zero, but the variable $x_k(t)$ remains tied, since it satisfies $x_k(t)=0,\,0<t<t_1$.

\end{compactdesc}

 \paragraph{Collisions that involve boundary values}
\begin{compactdesc}
\item[(d)] {\bf Interval collision between adjacent bases or at 0 or $T$,  
variable become free.}
This is the same as (c) except that the variable $x_k$ (or  $q_j$) become free in the sense that in the interior of the validity region $\V$,  in the time intervals $(0,t_{n_*})$ and $(t_{n_{**}}, T)$,  $x_k(t) >0$ (or $q_j(T-t)>0$)
(see Figure \ref{fig.cold}).  
\begin{figure}[H] 
\centering 
\includegraphics[width=4.6in]{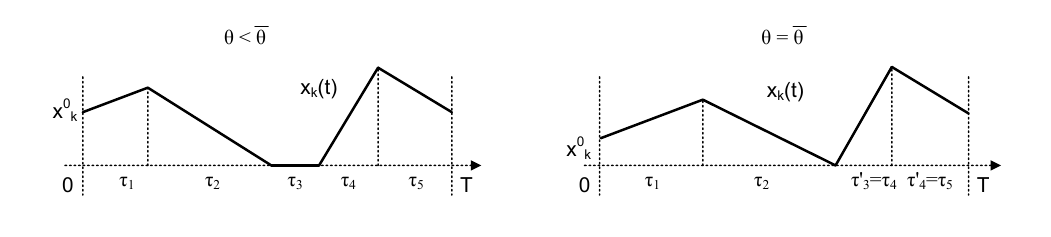} 
\caption[Type (d) collision, $\tau_3$ shrinks to $0$, $x_k$ becomes free.]{Type (d) collision, $\tau_3$ shrinks to $0$, $x_k$ becomes free.} 
\label{fig.cold} 
\end{figure}

\item[(e)]{\bf Boundary collision.}
One of the $>0$ boundary values including any of $\bu^0_j,\bu^N_j,\bp^0_k,\bp^N_k,\bx_k^N,$ $\bq_j^0$ shrinks to zero, or one of the values $x_k^0,\, k\in \K^{\uparrow\downarrow}$ or 
$q_j^N,\, j\in \J^{\uparrow\downarrow}$, shrinks to 0, or one of the values of $x_k^0,\, k\in \K^=$  for which $\dx_k^1 >0$ (see Figure \ref{fig.cole}) or one of the values of $q_j^N,\, j\in \J^=$  for which $\dq_j^N >0$ shrinks to 0.  
\begin{figure}[H] 
\centering 
\includegraphics[width=4.6in]{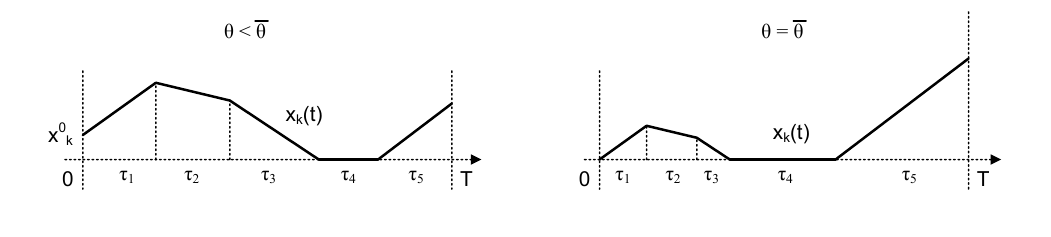} 
\caption[Type (e) collision, $x^0_k$ shrinks to $0$.]{Type (e) collision, $x^0_k$ shrinks to $0$.} 
\label{fig.cole} 
\end{figure} 

\item[(f)]{\bf Joint collision.}
One of the values of $x_k^0,\, k\in \K^=$  for which $\dx_k^1  <  0$  shrinks to 0, and as a result the intervals $\tau_1,\ldots,\tau_{n_{*}}$ shrink to zero, where in the interior of the validity region $x_k(t) > 0$ for $0 < t < t_{n_{*}}$, and $x_k(t_{n_{*}})=0$ (see Figure \ref{fig.colf}).  
\begin{figure}[h!t] 
\centering 
\includegraphics[width=4.6in]{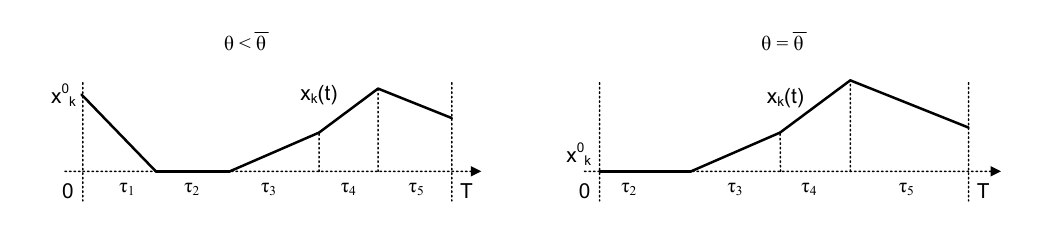} 
\caption[Type (f) collision, $x^0_k$ and $\tau_1$ shrink to $0$.]{Type (f) collision, $x^0_k$ and $\tau_1$ shrink to $0$.} 
\label{fig.colf} 
\end{figure} 
Alternatively,  one of the values of $q_j^N,\, j\in \J^=$  for which $\dq_j^N < 0$ shrinks to 0, and as a result the intervals $\tau_{n_{**}},\ldots,\tau_N$ shrink to zero, where in the interior of the validity region $q_j(T-t) > 0$ for $t_{n_{**}-1} < t < T$, and $q_j(T- t_{n_{**}-1})=0$.
\end{compactdesc}

\paragraph{Multiple collisions\\}
Any collision that is not one of the types (a)--(f) is {\em a multiple collision}.  Such collisions consists of a combination of several collisions of types (a)--(f). We will exclude the possibility of multiple collisions in Sections \ref{sec.pivot}--\ref{sec.example2}, but will return to discuss them in Appendix \ref{sec.general}.


\section{M-CLP pivots}
\label{sec.pivot} 
An M-CLP pivot uses the solution in validity region $\V$ with non-empty interior to construct the solution in a neighboring validity region with non-empty interior $\W$.  Let $\L(\theta),\,0<\theta<1$ be a parametric line of boundary parameters, such that $\L(\theta) \in \mbox{interior of } \V, \, \theta < \otheta$, $\L(\theta) \in \mbox{interior of } \W, \, \theta > \otheta$, and $\L(\otheta)$ is on the boundary of $\V$ and of $\W$. 
  We assume that at the collision point  $\L(\otheta)$,
 the collision from $\V$ as well as the collision from $\W$ are both of types (a)--(f).
%
In an M-CLP pivot we will start from the region $\V$ with optimal base sequence $(\K_n,\J_n)_{n=0}^{N+1}$, and with a specified type of collision, and we will show how to retrieve the collision from the other side, and from it construct the base sequence $(\K'_n,\J'_n)_{n=0}^{N'+1}$  optimal in $\W$.

There are three types of pivots:  
\begin{compactitem}[-]
\item
Pivots that involve only changes from $(\K_n,\J_n)_{n=1}^{N}$ to $(\K'_n,\J'_n)_{n=1}^{N'}$, which are pivots on the Internal-SCLP, similar to SCLP pivots described in \cite{weiss:08}.  We refer to these as internal pivots.
\item
Pivots that involve changes from $(\K_n,\J_n)_{n\in \{0,N+1\}}$ to  $(\K'_n,\J'_n)_{n\in \{0,N'+1\}}$ in which the solution at the collision point is unique.  We refer to these as boundary pivots of type I.
\item
Pivots that involve changes from $(\K_n,\J_n)_{n\in \{0,N+1\}}$ to  $(\K'_n,\J'_n)_{n\in \{0,N'+1\}}$ in which the solution at the collision point is not unique.  We refer to these as boundary pivots of type II.
\end{compactitem}

We now describe all the possible types of pivots:

\subsubsection*{Pure Internal-SCLP pivots}
\label{sec.int_pivot}
When the collision from validity region $\V$ is of types (a), (b), or (c) the pivot is an internal pivot.
\begin{compactitem}[-]
\item
Collision of type (a):  A single state variable  $x^n_k$, or $q^n_j$ shrinks to 0.  This is classified as  SCLP-III type collision in \cite{weiss:08}.  In this pivot new internal bases $D_1,\ldots,D_M$ are inserted  before $B_1$ (if $n = 0$), or after $B_{N}$ (if $n = N$), or between $B_n$ and $B_{n+1}$ (otherwise).  At the collision point these new bases have interval length 0, and they become $>0$ in the interior of $\W$.  The collision from $\W$ is of type (c) if the shrinking state variable is tied, or of type (d) if it is free. 
\item
Collision of type (b):  Intervals $\tau_{n^*},\ldots,\tau_{n^{**}},\,1<n^*\le n^{**} < N$ shrink to zero, with $|B_{n_*-1} \setminus B_{n_{**}+1} | =2$.  This is classified as  SCLP-II type collision in \cite{weiss:08}. 
The internal pivot involves removing the bases  $B_{n_*},\ldots,B_{n_{**}}$, and inserting new bases $D_1,\ldots,D_M,\,M\ge 1$ in their place.  At the collision point these new bases have interval length 0, and they become $>0$ in the interior of $\W$.  The collision from $\W$ is also of type (b).
\item
Collision of type (c):  Intervals $\tau_{n^*},\ldots,\tau_{n^{**}}$ shrink to zero, where either $1<n^*\le n^{**} < N$ with $|B_{n_*-1} \setminus B_{n_{**}+1} | =1$, or else $n^*=1$, or else  $n^{**}=N$, but no tied state variable becomes free.  This is classified as  SCLP-I type collision in \cite{weiss:08}.  The internal pivot involves removing the bases  $B_{n_*},\ldots,B_{n_{**}}$.  The collision from $\W$ is  of type (a).
\end{compactitem}

\subsubsection*{M-CLP pivots that involve boundary pivots}
\label{sec.bound_pivot}
This includes  collisions of type (d), (e), (f) from the $\V$ side,
and the pivot consists of up to three steps: it may require an initial Internal-SCLP pivot (pre-boundary step), it may then require a  pivot on the boundary dictionary (boundary step), and it may finally require another  Internal-SCLP pivot (post-boundary step).

The pre-boundary step is determined by the type of collision on the $\V$ side.  
Once pre-boundary SCLP pivot is complete the boundary pivot is carried out as follows:  First  the Boundary-LP simplex dictionary for $\V$ is constructed.  Then the leaving variable is determined according to the type (d), (e) or (f).  It is then determined if the pivot is type I or type II.  If it is type I no pivot is performed, but the bases  $(\K_n,\J_n)_{n\in \{0,N+1\}}$ may change.    If it is type II, a pivot of the boundary dictionary is performed, and the new bases $(\K'_n,\J'_n)_{n\in \{0,N'+1\}}$
are obtained.  This also determines the type of collision from the $\W$ side.  Finally, according to the type of collision from the $\W$ side a post-boundary SCLP pivot may be performed.  We now describe these steps.

\paragraph{The pre-boundary step}
\begin{compactitem}[-]
\item
Collision of type (d) or (f):  The bases of intervals that shrink to 0 are removed.
\item
Collision of type (e):  No pre-boundary pivot is required.
\end{compactitem}
We denote by $(\tcK_n, \tcJ_n)_{n=1}^{N''}$ the base sequence obtained from $(\K_n,\J_n)_{n=1}^{N}$ after the pre-boundary step, and refer to it as the {\em  intermediate base sequence}.

\paragraph{Constructing the Boundary-LP simplex dictionary for $\V$:}
In all three cases we follow the rules detailed in Section \ref{sec.dictionary}.  
In case (d) we need some additional rules:
If  $x_k$ became free then if  $ |\J^{\uparrow\downarrow}| + |\overline{\K_{N+1}}|+|\J_0| + |\overline{\K_0}|<2J$  we choose $\bp_k^0$ as a dual basic in $\bB^*$, and if  $ |\J^{\uparrow\downarrow}| + |\overline{\K_{N+1}}|+|\J_0| + |\overline{\K_0}|=2J$, we choose $x_k^\bullet$ as basic in $\bB$.  Similarly, if $q_j$ becomes free then if $|\K^{\uparrow\downarrow}| + |\overline{\J_0}| + |\K_{N+1}| + |\overline{\J_{N+1}}| <2K$ we choose $\bu_j^N$ as basic in $\bB$, and if  $|\K^{\uparrow\downarrow}| + |\overline{\J_0}| + |\K_{N+1}| + |\overline{\J_{N+1}}| =2K$  we choose $q_j^\bullet$ as dual basic in $\bB^*$.
We denote the Boundary-LP simplex dictionary for $\V$ by $\D$.

\paragraph{Determining the leaving variable}
In cases (e), (f) the variable leaving the basis is the boundary variable that has shrunk to 0 in the collision from $\V$.  We denote it by $\bv$, and it can be either a primal variable or a dual variable.  We denote its dual variable by $\bv^*$.  

In case (d), if $x_k$ has become free and $x_k^\bullet$ is basic, or if $q_j$ has become free and $q_j^\bullet$ is basic, then no variable leaves the basis and there is no pivot. We classify this case as a pivot of type I.
 Otherwise $\bv=\bp_k^0$ or $\bv=\bu_j^N$ respectively.   

\paragraph{Determine the type of boundary pivot and the entering variable}
We have the  following possibilities:  
\begin{compactitem}[-]
\item
If $\bv$ is a primal (or a dual) variable, examine the dual (or the primal) variables which are $\bw^* = 0$.  If any of them has a non-zero pivot element (i.e. if $\bv$ is in row $i$ and $\bw^*$ is in column $j$, the element $\hbA_{ij}$)  then this is a type I pivot.
%
\item
Otherwise the entering variable $\bw$ is determined, as in standard LP, by  the ratio test (this is the ratio test of the simplex method of linear programming, to determine which variable becomes tight and leaves the primal or the dual basis).
\begin{itemize}
\item
If there is a single candidate variable $\bw$ to enter with ratio $>0$ then we have a pivot of type II and $\bw$ is the entering variable.
\item
If there are multiple candidates to enter with ratio $>0$, then there are several variables shrinking to $0$ on the $\W$ side.  This indicates that there is a multiple collision  from the $\W$ side. We discuss multiple collisions in Appendix \ref{sec.general}.
\end{itemize}
\end{compactitem}

\paragraph{Boundary-pivot of type I:}
The solution at the collision point is unique, and there is no need for a pivot on the dictionary.  In the case of collision type (d), when $\bv=\bp_k^0$ (or $\bv=\bu_j^N$) we have $(\K_n,\J_n)_{n \in \{0, N+1\}}=(\K'_n,\J'_n)_{n \in \{0, N'+1\}}$, and the only change is that $x_k$ (or $q_j$) is no longer tied but becomes free in $\W$. 
 Otherwise (in case of collision type (e) or (f)), the solution at the collision point is unique, but $(\K_n,\J_n)_{n \in \{0, N+1\}}$ $ \ne (\K'_n,\J'_n)_{n \in \{0, N'+1\}}$.  
 
\paragraph{Boundary-pivot of type II:}
Assume $\bv$ leaving and $\bw$ entering.  Let $\bv^*$, $\bw^*$ be their dual variables.  Then the effect of the pivot is as follows:  Because the leaving variable $\bv=0$, after the pivot the new basic variable $\bw=0$.  However, since the test ratio is $>0$, the variable $\bw^*$ that leaves the dual basis is $>0$, and when it leaves the basis its value will jump down to 0.  Also, the variable entering the dual basis $\bv^*$ will jump from the non-basic value of 0 to a value $>0$.  In addition, all the basic boundary variables in the same category  as $\bw^*,\bv^*$ (that is in $\bB^*$ or in $\bB$ according to whether $\bv$ is primal or dual) will change their values.
\paragraph{Determine the variable shrinking to $0$ from the $\W$ side:}
From the $\W$ side we may have a single boundary variable shrinking to $0$, which we denote by $\bv'$, or there may be no boundary value that shrinks to 0 and we then define  $\bv' = \emptyset$.
\begin{compactitem}[-]
\item
 If the pivot is type I and the collision from $\V$ is (d) then $\bv' = \emptyset$.
\item 
If the pivot is type I and the collision from $\V$ is (e) or (f) then $\bv' = \bv^*$.
\item
In case of boundary pivot of type II,  if $\bw^*=x^\bullet_{k'}, \tx_{k'} > 0$ (or if $\bw^*=q^\bullet_{j'}, \tq_{j'} > 0$) then $\bv'= \emptyset$.
\item
Otherwise, for pivot of type II, $\bv' = \bw$.
\end{compactitem}
\paragraph{Determine the new boundary basis:}
To determine the new boundary bases $(\K'_n,$ $\J'_n)_{n \in \{0, N'+1\}}$ we first find the new values of $x^0, q^N$, from the new values of $x^\bullet$, $q^\bullet$ after the pivot, using $x^0 = x^\bullet+ \tx, q^N = q^\bullet + \tq$.  Recall that $\tx,\,\tq$ are uniquely determined at $\L(\otheta)$ by  Proposition \ref{thm.uniqueparams}. 

Following this step we define $(\K'_n,\J'_n)_{n \in \{0, N'+1\}}$ as follows:
\begin{equation}
\begin{array}{lllll}
j \in \J'_0  &  \text{if } \bq^0_j>0, \text{ or } \bv' = \bq^0_j, &                      \qquad k \in \K'_{N+1} &  \text{if } \bx^N_k>0, \text{ or } \bv' = \bx^N_k,  \\
k \in \K'_0 &  \text{if } x^0_k>0, \text{ or } \bv' = x_k^\bullet, &                          \qquad   j\in \J'_{N+1}  &  \text{if } q^N_j>0, \text{ or } \bv' = q_j^\bullet.
\end{array}
\end{equation}
This simply includes all $>0$ boundary values, and the variable $\bv'$ (if $\bv'\ne\emptyset$), and includes exactly $2(K+J)$ elements.
\paragraph{Test for a multiple collision from the $\W$ side:}
If $\bw^*=x_k^\bullet$ or  $\bw^*=q_j^\bullet$ and at  $\L(\otheta)$ the minimum of $x_k(t)$ or $q_j(t)$ is obtained in more than one point, as seen from $|\arg\min_{0 \le n \le N}  x^n_k|  >1$ (or from $|\arg\min_{0\le n \le N} q^n_j|  >1$), then there is a multiple collision from the $\W$ side.  This is illustrated in 
Figure \ref{fig.multiple_post} where $x_k^\bullet (\theta)$ denotes the value of the optimal $x_k^\bullet$ for parameters $\L(\theta)$.
\begin{figure}[h!t] 
\centering 
\includegraphics[width=4.6in]{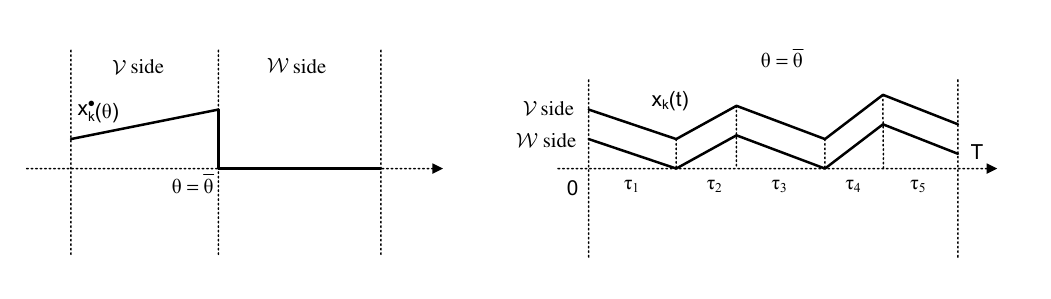} 
\caption{Boundary pivot type II that creates multiple collision: $\bw^*=x_k^\bullet$.} 
\label{fig.multiple_post} 
\end{figure} 

\paragraph{Determine the type of collision from the $\W$ side:} We assume that   no multiple collision is discovered in the previous step. 
\begin{compactitem}[-]
\item
If $\bv'=\emptyset$ and the boundary pivot is of type I, the collision from the $\W$ side is type (a).
\item
If $\bv'=\emptyset$ and the boundary pivot is of type II, the collision form the $\W$ side is type (d).
\item
If $\bv'=x^\bullet_{k'}$ and $\K'_0$ is not compatible with $\tcK_1$ (of the intermediate base sequence), or if $\bv'=q^\bullet_{j'}$ and $\J'_{N+1}$ is not compatible with $\tcJ_{N}$, then the collision from the $\W$ side is of type (f). 
\item
For any other $\bv' \ne \emptyset$, the  collision from the  $\W$ side  is of type (e)
\end{compactitem}

\paragraph{The post-boundary step}
In the post-boundary step we perform an SCLP pivot and build a base sequence $(\K'_n,\J'_n)_{n=1}^{N'}$ that is optimal on the $\W$ side.
\begin{compactitem}[-]
\item
If the collision from the $\W$ side is of type (d), and $\bw^*=x_k^\bullet$ and $t_n$ is the minimum point of $x_k(t)$ at $\L(\otheta)$, or $\bw^*=q_j^\bullet$ and $t_n$ is the minimum point of $q_j(T-t)$ at $\L(\otheta)$, the bases $D_1,\ldots,D_M$ need to be inserted between $\tilde{B}_n,\tilde{B}_{n+1}$, in the sequence $(\tcK_n,\tcJ_n)_{n=1}^{N''}$ as in SCLP-III type collision in \cite{weiss:08}. 
\item
If the collision from the $\W$ side is of type (f), then if $\bv'= x_k^\bullet$, then a number of bases $D_1,\ldots,D_M$ are inserted before $\tilde{B}_1$, and if 
$\bv'= q_j^\bullet$, then a number of bases $D_1,\ldots,D_M$ are inserted after $\tilde{B}_{N''}$, as in SCLP-IV type collision in \cite{weiss:08}. 
\end{compactitem}

\noindent
We now formulate the following two theorems that state the conditions under which  the pivot operation is well defined and constructs an optimal solution for the neighboring region $\W$.
\begin{theorem}
\label{thm.boundary-unique}
If the collisions on both the $\V$ and the $\W$ side are of types (a)--(f) then the  M-CLP pivot constructs a unique base sequence $(\K'_n,\J'_n)_{n=0}^{N'+1}$.
\end{theorem}
\begin{proof}  See Appendix \ref{sec.uniqueproof}  \rule{0.5em}{0.5em} \end{proof}

\begin{theorem}
\label{thm.facet}
If the collisions on both the $\V$ side as $\theta \nearrow \otheta$ and on the $\W$ side as $\theta \searrow \otheta$ are both of types (a)--(f) then 
 the base sequence $(\K'_n,\J'_n)_{n=0}^{N'+1}$ constructed by the M-CLP pivot is the optimal base sequence in $\W$.
\end{theorem}
\begin{proof}  See Appendix \ref{sec.uniqueproof}  \rule{0.5em}{0.5em} \end{proof}

\section{Algorithm under simplifying assumption}
\label{sec.algorithm_simpl}
Consider an M-CLP problem and let $\beta_g, \gamma_g,$ $ \lambda_g, \mu_g, T_g$ be a goal set of boundary parameters in the feasible parametric region, i.e. the set of boundary parameters for which an optimal solution should be found.  Let $\beta_0, \gamma_0, \lambda_0, \mu_0, T_0$ be another set of boundary parameters with a known optimal solution.
 Define {\em the parametric line  of boundary parameters} as $\L(\theta) = (1-\theta)(\beta_0, \gamma_0, \lambda_0, \mu_0, T_0) + \theta (\beta_g, \gamma_g, \lambda_g, \mu_g, T_g)$, $0 \le \theta \le 1$. 
The algorithm described in this section requires the following assumption:
\begin{assumption}[Simplifying Assumption]
\label{asm.smpl}
Assume that all collisions \hfill\linebreak along the parametric line $\L(\theta)$  are of type (a)-(f).  
\end{assumption}
We discuss the algorithm without this assumption in Appendix \ref{sec.general}

\subsection{Properties of the parametric line}
\label{sec.line}
\begin{theorem}
\label{thm.alg-line}
Let $\L(\theta), 0 \le \theta \le 1$ be a parametric line of boundary parameters.
If M-CLP/M-CLP$^*$ are feasible for both goal set of parameters $\L(1)$ and initial set of parameters $\L(0)$, then:

(i) all points of the parametric line  belong to the parametric-feasible region,

(ii) if $\L(0)$ is an interior point of some validity region, then all points of the parametric line except possibly $\L(1)$ belong to the interior of the parametric-feasible region,

(iii) the parametric line can be  partitioned to a minimal finite set of closed intervals of positive length $[\theta_{\ell -1}, \theta_\ell],$ $ \ell = 1,\dots, L \le \binom{4(K+J)}{2(K+J)}\; 2^{\binom{K+J}{K}}$ such that for each interval there is a (possibly non-unique)  proper base sequence  which is optimal for every $\L(\theta),\,\theta \in [\theta_{\ell -1}, \theta_\ell]$.
\end{theorem}
\begin{proof}  See Appendix \ref{sec.uniqueproof}  \rule{0.5em}{0.5em} \end{proof}

\begin{proposition}
\label{thm.alg-utheta}
Consider a continuous parametric line of boundary parameters $\L(\theta)$. Then the decomposition parameters $\tU(\theta), \tx(\theta), \tP(\theta), \tq(\theta)$ are continuous piecewise affine functions of $\theta$.
\end{proposition}
\begin{proof}
Within each closed validity region $\tU(\theta), \tx(\theta), \tP(\theta), \tq(\theta)$ are affine by Corollary \ref{thm.affine}.  Moreover, by Proposition \ref{thm.uniqueparams}, for every set of boundary parameters $\tU, \tx, \tP, \tq$ are unique, and therefore $\tU(\theta), \tx(\theta), \tP(\theta), \tq(\theta)$ are continuous at the boundary points of the validity regions $\L(\theta_\ell)$.
\qquad
 \rule{0.5em}{0.5em} \end{proof}

\subsubsection{Initial solution}
\label{sec.init}
The following theorem shows how to obtain a set of boundary parameters for a simple initial optimal solution.
\begin{theorem}
\label{thm.single-opt}
Consider a set of boundary parameters $\beta_0, \gamma_0, \lambda_0, \mu_0, T_0$ that satisfy:
\begin{equation}
\label{eqn.initial2}
 \beta_0 > 0, \; \lambda_0 < 0, \; \beta_0 + bT_0+ \lambda_0 > 0, \; \gamma_0 < 0, \;  \mu_0 > 0, \; \gamma_0 + cT_0 + \mu_0 < 0,
\end{equation}
then M-CLP/M-CLP$^*$ have a single-interval optimal solution, with
$U(t)=P(t)=0$, and $x(t),-q(t)$ equal to the right hand sides of 
(\ref{eqn.gmpclp}), (\ref{eqn.gmdclp}),  and  this set of boundary values is an interior point of the validity region.
\end{theorem}
\begin{proof}
It is easy to see that under (\ref{eqn.initial2}) $U(t)=0,\, x(t) = \beta_0 + bt,\, 0 \le t < T_0,\, x(T) = \beta_0 + bT_0 + \lambda_0,\, P(t) = 0,\, q(t) = - \gamma_0 - ct, 0 \le t < T_0, $ $q(T) = - \gamma_0 - cT_0 - \mu_0$ are feasible complementary slack solutions of M-CLP/M-CLP$^*$, and hence are optimal.
 They define an optimal base sequence $(\K_n, \J_n)_{n=0}^2, \, \K_0 = \K_1 = \K_2 = \{1,\dots,K\}, \, \J_0 = \J_1 = \J_2 = \{1,\dots,J\}$, with a single interval solution.
Moreover, one can see that the solution to the corresponding  Boundary-LP/LP$^*$ is unique, and hence by Theorem \ref{thm.corresp} this solution of M-CLP/M-CLP$^*$ is unique.  It is easy to see that $\bH_\setP >0$ and so by Theorem \ref{thm.interior}  the parameter values are an interior point of its validity region.
%
 \qquad
 \rule{0.5em}{0.5em} \end{proof}

\subsection{The algorithm}
\label{sec.algorithm}
We now describe a simplex-type algorithm to solve M-CLP/M-CLP$^*$ in a finite number of steps. For the algorithm to work we need the Non-Degeneracy Assumption \ref{asm.nondeg} and the Simplifying Assumption \ref{asm.smpl} to hold.

\subsubsection{Input} Problem data consists of $A, b, c$, and  goal  parameters $\L(1) = (\beta_{g}, \gamma_{g}, \lambda_{g}\le 0, \mu_{g}\ge 0, T_{g})$.

\subsubsection{Output} Optimal solution of  M-CLP/M-CLP$^*$ at $\L(1)$: optimal base sequence, breakpoint times, impulse controls at $0$ and $T_g$, rates $u, \dx, p, \dq$ in each interval, states $x(t), q(T_{g}-t)$ at each breakpoint.

\subsubsection{Feasibility test}
Solve Test-LP/LP$^*$ (\ref{eqn.ptestLP-ex}) for $\L(1)$ to determine whether both M-CLP/M-CLP$^*$ are feasible, or both are infeasible, or one is infeasible and the other unbounded. Proceed to solve the problems if both are feasible, otherwise stop. 

\subsubsection{Initialization}
Choose $\L(0)=(\beta_0,\gamma_0,\lambda_0,\mu_0,T_0)$ that satisfy:
\begin{equation*}
 \beta_0 > 0, \; \lambda_0 < 0, \; \gamma_0 < 0, \;  \mu_0 > 0, \; T_0 >0, \; \beta_0 + bT_0+ \lambda_0 > 0, \;  \gamma_0 + cT_0 + \mu_0 < 0.
\end{equation*}
Define $\delta\L=\L(1)-\L(0)$.  

The initial solution has  $U(t)=P(t)=0$, $x(t),-q(t)$ are r.h.s. of (\ref{eqn.gmpclp}), (\ref{eqn.gmdclp}) at  $\L(0)$, and 
the initial optimal base sequence is $(\K_n, \J_n)_{n=0}^2 = \{(\K_n = \{1,\dots,K\}, \J_n =  \{1,\dots,J\}), n = 0,1,2\}$.

Set:  $\ell := 1$,\, $\theta_\ell := 0$,\,  
$(\K,\J)^{(\ell)} :=  (\K_n, \J_n)_{n=0}^2$.

\subsubsection{Iteration}
\paragraph{$\bullet$ Set up equations\\}  Calculate rates for all bases in $(\K,\J)^{(\ell)}$, construct the coefficients of the matrix $M_\ell$ for $(\K,\J)^{(\ell)}$.   Construct $\bR_\ell$, $\delta \bR_\ell$ from $\L(\theta_\ell)$ and $\delta\L$.  

\paragraph{$\bullet$ Calculate current solution and its $\theta$ gradient\\}
Solve the equations
\[
M_\ell \bH^\ell = \bR_\ell,  \qquad  M_\ell \delta \bH^\ell = \delta \bR_\ell.
\]

\paragraph{$\bullet$ Find right endpoint of the validity region\\}
Find 
\[ 
1/\overline{\Delta} = \max_{\lbrace r: \bH_r^\ell \in \bH_\setP^\ell \rbrace}\left(0, \frac{-\delta \bH^\ell_r}{\bH^\ell_r}\right).
\] 
The right endpoint of the validity region is $\otheta : = \theta_\ell + \overline{\Delta}$.

\paragraph{$\bullet$ Termination test and conclusion\\}
If $\theta_\ell < 1 \le \overline{\theta}$ terminate: $(\K,\J)^{(\ell)}$ is the optimal base-sequence,  current rates are optimal, remaining output is obtained from $\bH = \bH^\ell + (1-\theta_\ell) \delta \bH^\ell$.

\paragraph{$\bullet$ Update solution\\}
Calculate values of the solution at $\otheta$, as $\overline{\bH} = \bH^\ell + \overline{\Delta} \delta\bH^\ell$. 
  
\paragraph{$\bullet$ Classify collision as $\theta \nearrow \otheta$\\}
Find the set  of elements of $H^\ell$ that are 0 at $\L(\otheta)$. Split $\K_0, \J_{N+1}$ to sets $\K^{=}, \K^{\uparrow\downarrow}, \J^{=}, \J^{\uparrow\downarrow}$.
Classify the collision to types (a)-(f) as discussed in Section \ref{sec.collisions}.

\paragraph{$\bullet$ Pivot \\}  Perform M-CLP pivot to determine a new base sequence, $(\K'_n,\J'_n)_{n=0}^{N'+1}$ as discussed in Section \ref{sec.pivot}.
Set $\ell := \ell+1$, $\theta_\ell := \otheta$, $(\K,\J)^{(\ell)}:= (\K'_n,\J'_n)_{n=0}^{N'+1}$ and move to the next iteration.

\subsection{Verification of the algorithm}
\label{sec.verify}
In this section we verify that under Simplifying Assumption \ref{asm.smpl} and Non-Degeneracy Assumption \ref{asm.nondeg}  all calculations described in Section \ref{sec.algorithm_simpl} are possible, and for any pair of feasible M-CLP/M-CLP$^*$ problems the algorithm produces an optimal solution in a finite bounded number of iterations. The main point here is to show that the new base sequence is  optimal for $\theta > \otheta$ small enough, and we show  that this property follows from Theorem \ref{thm.facet}.

\begin{theorem}
\label{thm.alg}
Assume that M-CLP/M-CLP$^*$ are feasible for the goal set of parameters $\beta_{g}, \gamma_{g}, T_{g},$ $  \lambda_{g}, \mu_{g}$. Assume also Non-Degeneracy Assumption \ref{asm.nondeg} and Simplifying Assumption \ref{asm.smpl}. Then the algorithm will compute the solution of M-CLP/M-CLP$^*(\beta_{g}, \gamma_{g}, T_{g}, \lambda_{g}, \mu_{g})$ in a finite number of iterations, bounded by $L \le 
\binom{4(K+J)}{2(K+J)}\;2^{\binom{K+J}{K}}$.
\end{theorem}
\begin{proof} The proof follows from the following sequence of Propositions. 
 \rule{0.5em}{0.5em} \end{proof}

\begin{proposition}
\label{thm.alg-feas}
The feasibility test can be performed, to determine whether M-CLP is feasible and bounded, or is infeasible, or is unbounded.
\end{proposition}
\begin{proof}
Clearly, one can  determine if Test-LP (\ref{eqn.ptestLP-ex}) is feasible and bounded, or is infeasible or is unbounded, which determines the same for M-CLP.   \qquad
 \rule{0.5em}{0.5em} \end{proof}
\begin{proposition}
\label{thm.alg-ini}
Initialization phase of the algorithm  yields a unique optimal solution with a unique optimal base sequence, and $\L(0)$ is an interior point of its validity region.
\end{proposition}
\begin{proof}
Follows from Theorem \ref{thm.single-opt}. \qquad
 \rule{0.5em}{0.5em} \end{proof}
\begin{proposition}
\label{thm.alg-term}
Assume that $\L(\theta)$ for some $\theta> \theta_\ell$ is in the interior of the validity region of $(\K,\J)^{(\ell)}$.  Then setting up equations, calculation of the current solution and its gradient, and calculation of the right endpoint $\otheta$, as well as the termination test can be carried out.
\end{proposition}
\begin{proof}
%
%
%

The base sequence $(\K,\J)^{(\ell)}$ determines the structure of $M_\ell$ and of $\bR_\ell$.  Solution of the Rates-LP/LP$^*$ for $(\K_n, \J_n)_{n=1}^N$ is unique by non-degeneracy assumption, and provides the non-zero coefficients of $M_\ell$. The non-zero coefficients of  $\bR_\ell$ are given by $\L(\theta_\ell)$.   $\delta \bR_\ell$ has the same structure as 
$\bR_\ell$,  and its non-zero elements are given by $\delta \L$.   
 
We assume in the statement of the proposition that the base sequence $(\K,\J)^{(\ell)}$ has non-empty interior, and hence by Proposition \ref{thm.nonsingular} the matrix $M_\ell$ is non-singular. Hence $H^\ell$ is the unique solution of $M_\ell \bH^\ell = \bR_\ell$ and $\delta H^\ell$ is the unique solution of $M_\ell \delta\bH^\ell = \delta\bR_\ell$.  We denote by $\bH(\theta)$ the solution at point $\L(\theta)$.
 Then at the point $\theta >\theta_\ell$ in the interior of the validity region the solution is determined by 
\[
M_\ell \bH(\theta) = M_\ell \big(\bH^\ell + (\theta-\theta_\ell)\delta \bH^\ell \big)  = \bR_\ell + (\theta-\theta_\ell) \delta \bR_\ell.
\]

That $\overline \Delta > 0$ follows because the interior of the validity region of 
$(\K,\J)^{(\ell)}$ is non-empty and contains an interval of the line $\L(\theta)$.  It is possible to have $ 1/\overline \Delta = 0$ and  $\overline \Delta=\infty$.


Finally, if any of $\delta \bH_r < 0$ then by $\bH_\setP(\theta) \ge 0$ it follows that $\bH_r > 0$, and the ratio $ - \delta \bH_r / \bH_r > 0$ and it is finite, and so $1/\overline \Delta > 0$ and finite so that $0 < \overline \Delta < \infty$.  Otherwise $1/\overline \Delta = 0$ and $\overline \Delta = \infty$.
\qquad
 \rule{0.5em}{0.5em} \end{proof}

\begin{proposition}
\label{thm.alg-pivot}
M-CLP pivot can be performed and is unique. There  exists $\theta > \otheta$ small enough, such that base sequence $(\K,\J)^{(\ell+1)}$ is optimal for this $\theta$ and $\L(\theta)$ is an interior point of its validity region.
\end{proposition}
\begin{proof}
Clearly $\otheta < 1$ and hence by Theorem \ref{thm.alg-line} $\L(\otheta)$ is an interior point of the parametric-feasible region. We prove by induction that for iteration $\ell$ the base sequence $(\K,\J)^{(\ell)}$ has a validity region with non-empty interior, i.e. there is a $\theta < \otheta$ big enough, such that the base sequence $(\K,\J)^{(\ell)}$ is optimal for this $\theta$ and $\L(\theta)$ is an interior point of its validity region. Clearly this assumption holds for  $\ell = 0$. Assume the induction hypothesis for $\ell$ which is optimal for $\V$.  Then since validity regions are closed (Theorem S3.7), there exists another validity region with non-empty interior $\W$, such that $\L(\otheta) \in \W$. Furthermore, by the Simplifying Assumption \ref{asm.smpl} the collision from $\W$ is of type (a)--(f) and hence by Theorem \ref{thm.facet} base sequence $(\K,\J)^{(\ell+1)}$ that is optimal for $\W$ can be obtained from $(\K,\J)^{(\ell)}$ by a single unique M-CLP pivot. Hence, there  exists $\theta > \otheta$ small enough, such that base sequence $(\K,\J)^{(\ell+1)}$ is optimal for this $\theta$ and $\L(\theta)$ is an interior point of its validity region $\W$.  This completes the proof.
 \rule{0.5em}{0.5em} \end{proof}

\begin{proposition}
The total number of iterations is bounded by $\left(\!\begin{smallmatrix}4(K+J) \\ 2(K+J) \end{smallmatrix}\! \right)2^{\binom{K+J}{K}}$.
\end{proposition}
\begin{proof}
Follows from Theorem \ref{thm.alg-line}.
\qquad
 \rule{0.5em}{0.5em} \end{proof}



\section{Illustrative Example Revisited}
\label{sec.example2}
We now return to the example problem (\ref{eqn.example}) solved in Section \ref{sec.example}, and examine it in view of the algorithm of Section \ref{sec.algorithm_simpl}.

We  first list all the admissible bases of the matrix $[A\;I]$.  
Each of the primal and dual bases is described by the rates of the  primal and dual state variables, which determine both the primal and the compatible dual bases. These in fact are the sets $\K,\J$.
\[
\begin{array}{l}
B_2=(\dx_1,\dq_2),\quad 
B_5=(\dx_2,\dq_1),\quad  
B_1=(\emptyset),\quad  
B_3=(\dx_2,\dq_2),\quad  
B_4=(\dx_1,\dq_1),\\ 
B_6=(\dx_1,\dx_2,\dq_1,\dq_2). 
\end{array}
\]
For the rate parameters  $b,c$ in (\ref{eqn.exampledata}), the objective values of the Rates-LP for these basic solutions are:
\[
V(B_2) = 9, \; V(B_5) = 8, \; V(B_1) = 7, \; V(B_3) = 6, \; V(B_4) = 6, \; V(B_6) = 0. 
\]
The resulting control rates and state rates for these basic solutions are given in Figure \ref{fig.rates}.
\begin{figure}[h!t]
\centering 
\includegraphics[width=4.6in]{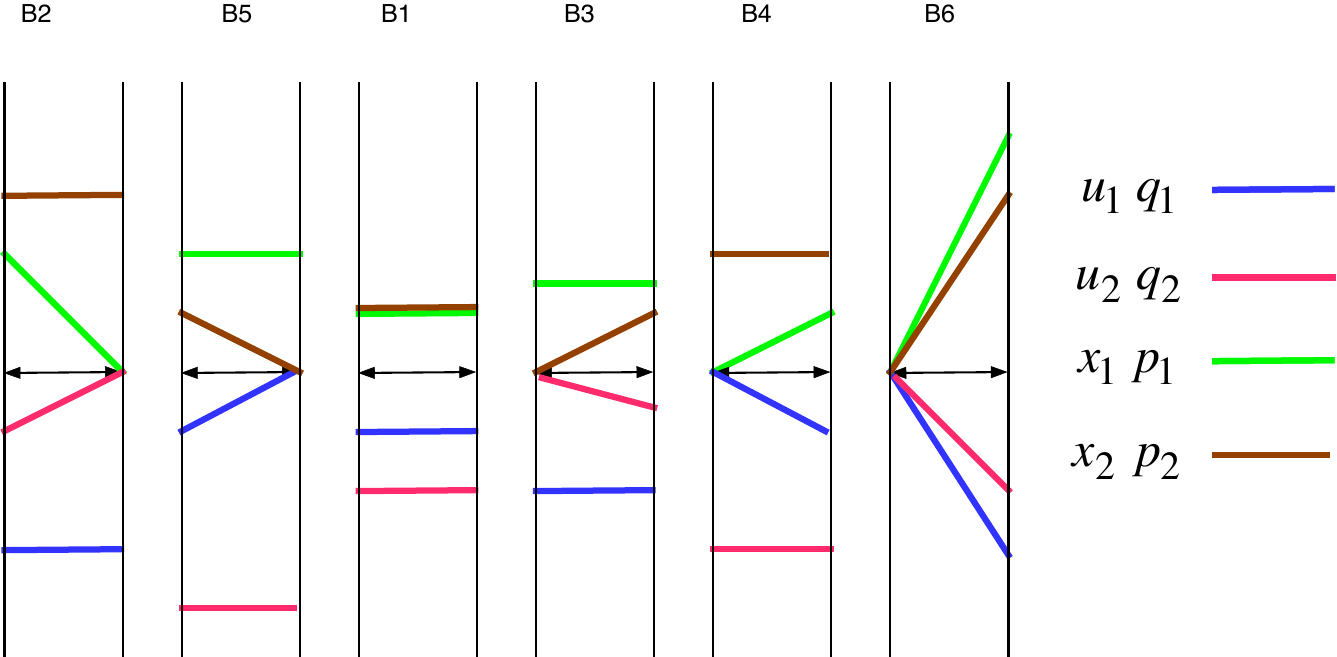} 
\caption{The primal and dual rates for all the admissible basic solutions of the Rates-LP.} 
\label{fig.rates}
\end{figure}

{\bf Note:}  It is never necessary to calculate all the basic solution of the Rates-LP with matrix $[A\;I]$.  The bases which are needed are obtained as necessary when internal pivots are performed, by pivoting from adjacent bases.  We list them here to show the rates, and we can do it in this example because it is of such small dimension.

The parametric line of boundary parameters for solving the problem is specified  in (\ref{eqn.exampleline}).

We now describe what happens at each of the iterations.  We will discuss mainly the boundary pivots. The internal pivots are according to the SCLP algorithm of \cite{weiss:08}.

\subsubsection*{Initialization}
The initial solution uses zero controls, $U(t)=P(t)=0$, and has states $x(t) = \beta+bt$, $q(t)=-\gamma-ct$.  The base sequence includes boundary values and the single internal basis $B_6$.  We write it as (here we list the internal bases, rather than the $\{\K_n,\J_n\}_{n=1}^N$):
\[
(\K, \J)^{(0)} = \{ \left[\begin{array}{c} x^0_1 \\ x^0_2 \\ \bq^0_1 \\ \bq^0_2 \end{array} \right]  , B_6 , 
\left[\begin{array}{c} \bx^N_1 \\ \bx^N_2 \\ q^N_1 \\ q^N_2 \end{array} \right]  \}
\]
Note that there is a jump in $x(t)$ at $T$:  $\bx^N - x^N =  \lambda  - A (U(T)-U(T-))$, and a similar jump in $q(t)$:  $\bq^0 - q^0 = - \mu  + A\Tt (P(T)-P(T-))$.  These jumps are  shown in Figure \ref{fig.p2iteration1}.
All the state variables in this solution are free.

\subsubsection*{First iteration}
At $\theta^{1}=1/11$ there is a collision of type (e), as $q_1(T)=\bq^0_1$ shrinks to 0.
This requires a boundary pivot.  No pre-boundary internal pivot is necessary.
At this point the boundary values are:
\[
\begin{array}{l}
x_1^\bullet = \frac{24}{11}, \; \tx_1=0, \; x_2^\bullet=1, \; \tx_2=0,\quad
q_1^\bullet = \frac{10}{11} , \; \tq_1 = \frac{39}{11} ,\;  q_2^\bullet = 2,  \; \tq_2 = \frac{26}{11}, \\
\\
\bx^N_1 = 6 ,\; \bx^N_2 = \frac{40}{11},\; \qquad \bq^0_1 = 0,\;  \bq^0_2 = \frac{12}{11}.
\end{array}
\]
We can now formulate the Boundary-LP problems (\ref{eqn.modprimalpbLP}), (\ref{eqn.moddualpbLP}), and construct the dictionary according to the rules in Section \ref{sec.dictionary}.  In the dictionary, $\bq^0_1$ is 0 and needs to leave the dual basis, i.e. $\bu^0_1$ will enter the primal basis.  Examination of the dictionary shows that $x_2^\bullet$ should leave the primal basis,  i.e.  $\bp^0_2$ will enter the dual basis.  Since the test ration is $\ne 0$,  this is a boundary pivot of type II.  

The {\em boundary dictionaries}, before and after the pivot are:
\[
\D^{(0)} = \begin{array}{r|c|cccc}
 & &  \uparrow & & & \\
 & &  \bq^0_1 & \bq^0_2 & q^\bullet_1 & q^\bullet_2 \\
 \hline
 & & 0 & \frac{12}{11} & \frac{10}{11} & 2 \\
 \hline
 x^\bullet_1 &  \frac{24}{11} & 2  & 1 & 0 & 0 \\
\to x^\bullet_2 &  1 & {\color{red}\bf 1} & 1 & 0 & 0  \\
 \bx^N_1 &  6 & 2 & 1 & 2 & 1 \\
 \bx^N_2 &  \frac{40}{11} & 1 & 1 & 1 & 1 
\end{array}\;,
\qquad
\D^{(1)} = \begin{array}{r|c|cccc}
 & &   & & & \\
 & &  \bp^0_2 & \bq^0_2 & q^\bullet_1 & q^\bullet_2 \\
 \hline
 & & 0 & \frac{12}{11} & \frac{10}{11} & 2  \\
 \hline
 x^\bullet_1 &  \frac{2}{11} & -2 & -1 & 0 & 0 \\
 \bu^0_1 & 1 & 1 & 1 & 0 & 0  \\
 \bx^N_1 &  4 & -2 & -1 & 2 & 1  \\
 \bx^N_2 &  \frac{29}{11} & -1 & 0 & 1 & 1 
\end{array}
\]
Following this pivot, in the interval of  $(\theta^{(1)},\theta^{(2)})=(1/11,1/6)$ there are two impulse controls:
$\bu^0_1$, and $\bp^0_2$,  which replaced the boundary values $\bq^0_1$ and $x^0_2$.

\subsubsection*{2nd to 5th iterations}
These are all internal pivots.  

At $\theta =1/6$ there is a collision of type (a) at $t_0$, where $q^0_1=q_1(T-)$ shrinks to zero.   
In the internal pivot the internal base sequence is augmented, from $B_6$ to $B_2,B_6$.  
In the collision from the other side, $q_1(t)$ becomes free (it is free for $\theta < 1/6$ and tied for $\theta > 1/6$), so it is a type  (d) collision form the other side.

At $\theta =2/9$, there is a collision of type (a) at $t_1$, where $x^1_1=x_1(t_1)$ shrinks to zero.  In the internal pivot the internal base sequence is augmented, from $B_2,B_6$ to $B_2,B_3,B_6$.    The collision from the other side is again type (d). 

At $\theta =4/9$, there is a collision of type (a) at $t_1$, where $q^1_2=q_2(T-t_1)$ shrinks to zero.  
 In the internal pivot the internal base sequence is augmented, from $B_2,B_3,B_6$ to $B_2,B_1,B_3,B_6$.  The collision from the other side is type (d).

At $\theta =5/9$, the interval $(t_2,t_3)$ with the rates basis $B_3$ shrinks to 0, 
and $|B_1 \backslash B_6| = |\{u_1 ,u_2 \} \backslash \{\dx_1,\dx_2\}| = 2$, so this is a type (b) internal collision.    In the internal pivot, $B_3$ is deleted from $B_2,B_1,B_3,B_6$, and then $B_4$ is inserted in its place, so the new internal base  sequence is $B_2,B_1,B_4,B_6$.  
The collision form the other side is also type (b).

\subsubsection*{6th iteration}
In the interval $\theta > 5/9$ the optimal base sequence is
\[
(\K, \J)^{(5)} = \{ \left[\begin{array}{c} x^0_1 \\ \bu^0_1 \\ \bp^0_2 \\ \bq^0_2 \end{array} \right]  , B_2,B_1,B_4,B_6, 
\left[\begin{array}{c} \bx^N_1 \\ \bx^N_2 \\ q^N_1 \\ q^N_2 \end{array} \right]  \}
\]
All three boundary states in $\K_0,\J_{N+1}$, namely $x_1(t),q_1(t),q_2(t)$, are now tied. 

At $\theta = 13/19$, $\bx^N_2$ shrinks to zero, which is a collision of type (e), and a boundary pivot is needed.  There is no pre-boundary internal pivot.  The boundary values at this point are: 
\[
\begin{array}{l}
x^\bullet_1 = 0,\; \tx_1=\frac{26}{19},\; \bu^0_1 = 1, \quad 
q^\bullet_1 = 0,\; \tq_1= \frac{17}{19},\; q^\bullet_2 = 0,\; \tq_2= \frac{4}{19}, \\
\\
\bp^0_2 = \frac{6}{19},\;  \bq^0_2 = \frac{13}{19},\; \quad \bx^N_1= \frac{13}{19},\; 
\bx^N_2 = 0.
\end{array}
\]
The dictionary for the pivot is:
\[
\D^{(5)} = \begin{array}{r|c|cccc}
 & &   & & & \\
 & &  \bp^0_2 & \bq^0_2 & q^\bullet_1 & q^\bullet_2 \\
 \hline
 & & \frac{6}{19} & \frac{13}{19} & 0 & 0 \\
 \hline
 x^\bullet_1 &  0 & -2 & -1 & 0 & 0 \\
 \bu^0_1 & 1 & 1 & 1 & 0 & 0  \\
 \bx^N_1 &   \frac{13}{19} & -2 & -1 & 2 & 1  \\
 \bx^N_2 &  {\color{red}0} & -1 & 0 & 1 & 1 
\end{array}
\]
The variable that shrunk to 0 is $\bx^N_2$, and it will leave the primal basis.
We now need to decide on the type of pivot.  In this case, the dual boundary values include $q^\bullet_1 = q^\bullet_2 =0$, and both have non-zero pivot elements, so this is a boundary pivot of type I.  There is no need to perform a pivot operation on $\D^{(5)}$.

The variable that shrinks to zero on the other side, i.e. from $\theta >  13/19$ is the dual variable to $\bx^N_2$, which is $\bp^N_2$.  The new optimal base sequence is:
\[
(\K, \J)^{(6)} = \{ \left[\begin{array}{c} x^0_1 \\ \bu^0_1 \\ \bp^0_2 \\ \bq^0_2 \end{array} \right]  , B_2,B_1,B_4,B_6, 
\left[\begin{array}{c} \bx^N_1 \\ \bp^N_2 \\ q^N_1 \\ q^N_2 \end{array} \right]  \}
\]

\subsubsection*{Last iteration}
This iteration starts with the optimal base sequence $(\K, \J)^{(6)}$, and this remains optimal for all $13/19 < \theta < 1$.  In the interior of this region there are 4 intervals, and there are three impulse controls,  $\bu^0_1,\, \bp^0_1,\, \bp^N_2$.  When $\theta$ reaches 1, a multiple collision occurs, as several items shrink to zero.  This is typical, because at that point we have $\lambda=\mu=0$, so the solution of the boundary problem is degenerate.  The things that happen at $ \theta = 1$ are:  the interval $(t_4,T)$ shrinks to zero, the impulse control $\bp^0_2$ shrinks to zero,  $x^N_2$ shrinks to zero, and $q^N_2$ shrinks to zero.   

The final solution of the problem consists of three intervals, and two impulse controls, $\bu^0_1$ at primal time 0, and $\bp^N_2$ at dual time 0.  This is the solution displayed in Figure  \ref{fig.p2full}.

\appendix

\section{Appendix}
\subsection{Algorithm under general settings}
\label{sec.general}
Under the Simplifying Assumption \ref{asm.smpl}, all collisions were single.  We now discuss the algorithm without the simplifying assumption.  The main idea is that whenever a multiple collision is discovered in step $\ell$ of the algorithm, going from $\L(\theta_\ell)$ to $\L(\theta_{\ell+1})$  through an interior of validity region $\V$, one moves to an interior point of $\V$ that is close to the current parametric line, and one restarts from this revised  point on a new parametric line leading to $\L(1)$.  

The general algorithm works exactly as the simplified algorithm described in Section \ref{sec.algorithm_simpl} when collisions at point $\L(\theta_{\ell+1})$ is of type (a)-(f) from both $\V$ and $\W$ sides. However, if collision from the $\V$ or the $\W$ side is a multiple collision there are some additional steps that build a new parametric line $\L'(\theta)$.  

We extend the collision classification given in Section \ref{sec.collisions} by the following types:
\begin{itemize}
\item 
{\bf Pre-$\otheta$ multiple collision.} This is a multiple collision, that is discovered when the solution for the endpoint $\otheta$ is calculated.  It is then seen if the set of values of $\bH_\setP^\ell$ that shrink to zero
\[
\M = \argmax_{\lbrace r: \bH_r^\ell \in \bH_\setP^\ell \rbrace}\left(0, \frac{-\delta \bH^\ell_r}{\bH^\ell_r}\right)
\]
indicates a multiple collision. 
\item
{\bf At-$\otheta$ multiple collision.} 
If there is no pre-$\otheta$ collision, and there is a collision of types (d)-(f) from the $\V$ side, but in the boundary pivot in which $\bv$ leaves the boundary basis there are several boundary variables with equal ratio $>0$ that are candidates to enter the basis, this is an at-$\otheta$ multiple collision. In that case, these candidates to enter are a set of boundary variables that shrink to zero from the $\W$ side as $\theta \searrow \otheta$. The ratio test, that is a part of boundary pivot discussed in Section \ref{sec.bound_pivot}, performs discovery of this collision.
\item 
{\bf Post-$\otheta$ multiple collision.}
If there is no pre-$\otheta$ or at-$\otheta$ multiple collision, 
then there could be a multiple collision from the $\W$ side in which several local minima of the state variables shrink to zero. This can happens when there is a pivot of type II, and $\bw^*=x_k^\bullet$ or $\bw^*=q_j^\bullet$, and it is discovered by the test for multiple collision from $\W$ side discussed in Section \ref{sec.bound_pivot}.
%
%
%
\item
{\bf Single collision.} 
This is a collision that occurs under the simplifying assumption, i.e. it is a collision of one of the types (a)-(f) from both the $\V$ and the $\W$ side. 
\end{itemize}
It is easy to see that these collisions cover all possible collisions that may occur along the parametric line.

For all these types of collisions one can choose an interior point $\L(\theta) \in \V$, a direction that is orthogonal to the current parametric line, denoted by  $\vL^{\bot}$, and a step size $\epsilon$ and build a new parametric line $\L'(\theta)$ that goes through the two points $\L(\theta) + \epsilon \vL^{\bot}$ and $\L(1)$ and this line can be chosen to satisfy the  following properties:
\begin{itemize}
\item 
$\L(\theta) + \epsilon \vL^{\bot}$ is an interior point of $\V$, and the line reaches the boundary of $\V$ at a $\otheta'$.
\item Parametric line $\L'(\theta)$ goes through all the validity regions that were crossed by the parametric line $\L(\theta)$, from the point $\L(0)$ up to the point $\L(\otheta)$.
\item 
If there is a pre-$\otheta$ multiple collision at $\L(\otheta)$ from the $\V$ side, then there is either at-$\otheta$, post-$\otheta$ or a single collision at $\L'(\otheta')$.
\item If there is an at-$\otheta$ multiple collision at $\L(\otheta)$ from the $\V$ side, then there is either post-$\otheta$ or a single collision at $\L'(\otheta')$.
\item 
If there is a post-$\otheta$ multiple collision at $\L(\otheta)$ from the $\V$ side, then there is a single collision at $\L'(\otheta')$.
\end{itemize}
One can see that under such a policy the algorithm cannot cycle and it will still find an optimal solution for the goal set of boundary parameters at the point $\L(1)$ that is shared by all the parametric lines in a finite number of iterations bounded by $\binom{4(K+J)}{2(K+J)}\;2^{\binom{K+J}{K}}$. 

We leave the concrete calculations of an interior point $\L'(\theta) \in \V$, an orthogonal direction $\vL^{\bot}$ and a step size $\epsilon$ that satisfy  the requirements of this policy   out of this paper.   We state that this can be done, but it will be more appropriate to deal with the details when one creates an implementation of the algorithm.  For further discussion of the requirement of perturbations see \cite{shindin:16}

\subsection{Discussion of $\lambda$ and $\mu$}
\label{sec.lambdamu}

The motivation for including $\lambda,\,\mu$ in the M-CLP/M-CLP$^*$ formulation is that otherwise (i.e. when $\lambda=\mu=0$), solutions in the interior of validity regions may not have $H_{\setP}>0$, and validity regions may not have disjoint interiors.  

Strong duality and structure properties of M-CLP/M-CLP$^*$ with $\lambda=\mu=0$ remain valid for solutions of $\lambda,\mu \ne 0$, except for right continuity of solutions at $t=0$.  It is possible that optimal solutions now  have $U(0-)=0 < U(0) < U(0+)$ (or  $P(0-)=0 < P(0) < P(0+)$).  

However, if we restrict the formulation to $\lambda \le 0$, $\mu \ge 0$, this cannot happen as we state in the next theorem.

\begin{theorem}
For feasible M-CLP/M-CLP$^*$ with $\lambda \le 0, \mu \ge 0$, there  exists a pair of optimal solutions $U_0(t), P_0(t)$ that are right-continuous at $0$.
\end{theorem}
\begin{proof}
Assume first  that the Slater type condition of \cite{shindin-weiss:13} holds.  Imitating the steps leading to Theorem D4.7 and D5.5(iii), M-CLP/M-CLP$^*$ of (\ref{eqn.gmpclp}), (\ref{eqn.gmdclp}) possess strongly dual optimal solutions $U_*(t), P_*(t)$, which are continuous piecewise linear on $(0,T)$.
  What remains to be checked is whether  $U_*(t),P_*(t)$ are right continuous at $0$.  
Let $\bu_*^0 = U_*(0)$ and $\bu_*^{0+} = U_*(0+) - U_*(0)$. Consider the following solution of M-CLP:
\[
U_0(t)=\left\{ \begin{array}{ll} 
\displaystyle \bu_*^0 + \bu_*^{0+} = U_*(0+), & t = 0,  \\
\displaystyle U_*(t), & t > 0.
\end{array} \right.
\]
We check  that $U_0(t)$ is a feasible solution of M-CLP at $t=0$. By feasibility of $U_*(t)$  we have $A U_0(0) = A U_0(0+) = A U_*(0+) \le \beta$, and therefore $U_0(t)$ is a feasible solution of M-CLP. Next we compare the objective values produced by $U_*(t)$ and $U_0(t)$.
\[
\begin{array}{c}
\displaystyle \mu\Tt U_0(0) + \int_{0-}^T (\gamma+ (T-t)c)\Tt dU_0(t) - \mu\Tt U_*(0) - \int_{0-}^T (\gamma+ (T-t)c)\Tt dU_*(t) = \\ 
\displaystyle = \mu\Tt (\bu_*^0 + \bu_*^{0+}) + (\gamma + cT)\Tt (\bu_*^0 + \bu_*^{0+}) + \int_{0+}^T (\gamma+ (T-t)c)\Tt dU_*(t) - \\
\displaystyle - \mu\Tt \bu_*^0 - (\gamma + cT)\Tt \bu_*^0 - (\gamma + cT)\Tt \bu_*^{0+} - \int_{0+}^T (\gamma+ (T-t)c)\Tt dU_*(t) = \mu\Tt \bu_*^{0+} \ge 0.
\end{array}
\]
Hence, the right-continuous $U_0(t)$ is an optimal solution of M-CLP.  The construction of $P_0(t)$ is similar.  Note that $\mu \ge 0$ (and similarly $\lambda \le 0$) is indeed necessary for the proof.
To complete the proof we use the detailed structure of the solution described in \cite{shindin-weiss:14}, where the second boundary equations (S11) have been changed to the equations (\ref{eqn.boundary2}), and  use Theorem S4.1 to show that the Slater-type condition is not necessary. \qquad
 \rule{0.5em}{0.5em} \end{proof}

It is easy to check that  all the results of \cite{shindin-weiss:13,shindin-weiss:14} hold for the extended formulation with $\lambda \le 0$ and $\mu \ge 0$.  Further discussion of the motivation and examples are included in \cite{shindin:16}.

\subsection{ Proof of Theorem \ref{thm.interior}}
\label{sec.validityappendix}
\begin{proof}
{\em (i) $\implies$ (iii):} Let $\rho$ be an interior point of $\V$ and let $M$, $\bR$ be the corresponding matrix and r.h.s. of (\ref{eqn.ex-coupled}) for the base sequence $(\K_n, \J_n)_{n=0}^{N+1}$ and the point $\rho$. Assume  contrary to the theorem that the solution of (\ref{eqn.ex-coupled}) is not unique. 

 
We now consider the solution of  (\ref{eqn.ex-coupled}), to see which components may be  not unique.  By Theorem D5.5(iii), all $\tau_n$ and $x_k^0,\,k\in \K^=,\,q_j^N,\,j\in \J^=$ are unique.  All the boundary values listed in Equations (\ref{eqn.compatible}) are uniquely determined to be 0.  Also, all of $x_k(t),\,k\in \K^=$ and $\,q_j(t),\,j\in \J^=$ are uniquely determined by the unique values of the $x_k^0,\,k\in \K^=,\,q_j^N,\,j\in \J^=,\, \tau_n$.  Hence, only components of $\bu^0,\bu^N,\bp^0,\bp^N,\bx^N,\bq^0$ and $x_k^0,\,k \in \K^{\uparrow\downarrow},\,q_j^N,\,j \in \J^{\uparrow\downarrow}$ and $x_k(t),q_j(t),\, k\in \K^{\uparrow\downarrow},j\in \J^{\uparrow\downarrow}$ may be non-unique.  By (\ref{eqn.lineq2}), if $x_k(t),q_j(t),\, k\in \K^{\uparrow\downarrow},j\in \J^{\uparrow\downarrow}$ are non-unique then $x_k^0,\,k \in \K^{\uparrow\downarrow},\,q_j^N,\,j \in \J^{\uparrow\downarrow}$ must be non-unique.  

We now note that $\bu^0,\bu^N,\bp^0,\bp^N,\bx^N,\bq^0$ and $x_k^0,\,k \in \K^{\uparrow\downarrow},\,q_j^N,\,j \in \J^{\uparrow\downarrow}$ are determined by equations 
(\ref{eqn.boundary1}) and by equations (\ref{eqn.boundary2}) after substitution of (\ref{eqn.lineq2}).  But these equations have $\beta,\gamma,\lambda,\mu$ on the r.h.s..

If the solution is non-unique then $M$ must be singular, and so we must have:  $M_{r'} = \alpha_{r_1} M_{r_1} + \dots + \alpha_{r_m} M_{r_m}$, for which also $\bR_{r'} = \alpha_{r_1} \bR_{r_1} + \dots + \alpha_{r_m} \bR_{r_m}$ for some of the rows of $M$.  But by the above argument, all these rows must have some component of $\beta,\gamma,\lambda,\mu$ on the r.h.s..  

  Consider $\delta \rho$ in which we change some of the values of  
  $\beta,\gamma,\lambda,\mu$ in such a way that $\delta\bR_{r'} \ne \alpha_{r_1} \delta \bR_{r_1} + \dots + \alpha_{r_m} \delta \bR_{r_m}$.  One can see that for any $\Delta > 0$ the equations (\ref{eqn.ex-coupled}) with r.h.s. $\bR + \Delta \delta \bR$ have no solution, and therefore $\rho$ cannot be an interior point. 
 This shows that (i) implies that the solution must be unique.
 
 We now consider the components of $\bH_\setP$.  Look first at  $\tau_n$ and $x_k(t),\,k\in \K^=,\,q_j(t),\,j\in \J^=$.  Any of the components of these which belong to $\bH_\setP$ can be changed by choosing an appropriate change in $\beta,\gamma$ or $T$.
By the above discussion, all the remaining components of  $\bH_\setP$ can be changed by appropriate choice of $\beta,\gamma,\lambda,\mu$.  Assume now that for  some element  $H_r$ of $\bH_\setP$, the value is $H_r=0$.  Then we can choose $\delta \rho$ such that for the corresponding $\delta \bR$ we have in the solution of $\delta \bR = M \delta H$ that $\delta H_r \ne 0$.  But in that case the solution of (\ref{eqn.ex-coupled}) for $\rho + \Delta \delta \rho$ will have $H_r<0$ for all $\Delta > 0$ or for all $\Delta <0$, in contradiction to  the assumption that  $\rho$ is an interior point.  This proves that (i) implies that all the components of $\bH_\setP$ are positive. 

{\em (iii) $\implies$ (ii):} Consider $\rho \in \V$ and let $M$, $\bR$ be the corresponding matrix and r.h.s. of (\ref{eqn.ex-coupled}) for the base sequence $(\K_n, \J_n)_{n=0}^{N+1}$ and the point $\rho$. Assume  that the solution of (\ref{eqn.ex-coupled}) for $\rho$ is unique and satisfies $\bH_\setP>0$. Hence, for this base sequence the boundary values are unique and strictly complementary slack, which implies the  uniqueness of $(\K_n, \J_n)_{n\in \{0,N+1\}}$. Moreover, all interval lengths are $>0$ and all values of $x^n,q^n$ that are members of $H_\setP$ are $>0$, and hence at each $t_n, n=1,\ldots,N-1$ there is exactly one state variable that can leave or enter the basis, and so there cannot be another sequence of bases with zero length intervals between any of $B_1,\ldots,B_N$.  Furthermore, at time $T$, none of the variables $x_k^N,\,k\in\K_N$ is $0$, so no new intervals of length 0 are possible after $B_N$, and similarly before $B_1$.  Hence base sequence $(\K_n, \J_n)_{n=0}^{N+1}$ is unique.

{\em (ii) $\implies$ (i):}  
Assume $\rho \in \V$ for the unique base sequence $(\K_n, \J_n)_{n=0}^{N+1}$. Assume contrary to the theorem that $\rho$ is a boundary point of the $\V$. Then, as $\rho$ is an interior point of the parametric-feasible region, from the convexity of the parametric-feasible region and closedness of validity regions (Theorem S3.7) it follows that $\rho$ belongs to a boundary of another validity region $\V'$. This contradicts  the uniqueness of the $(\K_n, \J_n)_{n=0}^{N+1}$ and hence $\rho$ is an interior point of the $\V$.
 \rule{0.5em}{0.5em} \end{proof}

\subsection{Proofs for Section \ref{sec.pivot}}
\label{sec.uniqueproof}
\begin{proof}[Proof of Theorem \ref{thm.boundary-unique}]
The rules for the M-CLP pivot define a unique  $(\K'_n,\J'_n)_{n=0}^{N'+1}$, if we can show the following:
\begin{compactenum}[(i)]
\item
There is no multiple collision discovered from the $\W$ side.
\item
If the boundary dictionary on the $\V$ side is not unique, then the results of the M-CLP pivot do not depend on the choice of boundary dictionary.
\end{compactenum}
That (i) holds is true by the assumption that the collision from $\W$ is of type (a)--(f) i.e. it is not a multiple collision.

We now show (ii). 
Consider the case that the leaving variable is a primal variable $\bv=\bv_i=0$ (in row $i$ of the dictionary $\D$).  Assume first that for all the dual variables $\bw_l^*$ such that $\bw_l^*=0$, also $\hbA_{i,l}=0$.  If $\D$ is not a unique dictionary then we can pass from $\D$ to any other dictionary $\D''$ by a series of pivots in which a primal variable with value 0 is exchanged for a dual variable of value 0.  In that case there must be a pair $\bv_k=0$ and $\bw_l^*=0$ such that $\hbA_{k,l} \ne 0$.  Otherwise the dictionary is unique.  Consider then the pivot from $\D$ to $\D'$ obtained by pivoting on  $\hbA_{k,l} \ne 0$.   We assumed that $\hbA_{i,l}=0$, and as a result, the pivot will not change row $i$.  It will also not change the values of the dual variables.  Hence the ratio test will yield the same $\bw^*_r>0$ that will leave the dual basis, with $\bw_r$ entering the primal basis, for all possible dictionaries.

Assume now that for some $\bw^*_j=0$ the element $\hbA_{i,j}\ne 0$.  Then by the above argument, for any other possible dictionary $\D''$, there will be some 
$\bw^*_l=0$ such that $\hbA_{i,l}\ne 0$.  But in that case it follows that the pivot under all possible dictionaries will be a pivot of type I.  But in that case all that will happen is that $\bv$ will leave the primal basis, and $\bv^*$ will enter the dual basis, independent of the dictionary chosen.
 \rule{0.5em}{0.5em} \end{proof}

\begin{proof}[Proof of Theorem \ref{thm.facet}]
We distinguish two possibilities.  
\begin{compactitem}[-]
\item
Case (1):  If there are only two optimal base sequences at $\L(\otheta)$, then $(\K_n,\J_n)_{n=0}^{N+1}$ is optimal in $\V$ and 
$(\K'_n,\J'_n)_{n=0}^{N'+1}$ is optimal in $\W$.
\item
Case (2):  There is another (or several) optimal base sequence, $(\K''_n,\J''_n)_{n=0}^{N''+1}$ which is optimal at $\L(\otheta)$.  In that case  we will show that it has a validity region with an empty interior, so again we  have that  $(\K_n,\J_n)_{n=0}^{N+1}$ is optimal in $\V$ and 
$(\K'_n,\J'_n)_{n=0}^{N'+1}$ is optimal in $\W$.
\end{compactitem}
We now consider all types of M-CLP pivots, and check if they fall into case (1) or case (2), and in case (2) we show that all other optimal sequences have empty validity regions.

If the pivot is an Internal-SCLP pivot then all the boundary values are positive, so $(\K'_n,$ $ \J'_n)_{n \in \{0,N'+1\}} = (\K_n,\J_n)_{n \in \{0,N+1\}}$ is unique.  Furthermore, by \cite{weiss:08}, if the collision is not a multiple collision then $(\K_n,\J_n)_{n=1}^{N}$ and $(\K'_n,\J'_n)_{n=1}^{N'}$ are the only two optimal base sequences at $\L(\otheta)$.  Hence we are in case (1).  

If the pivot is a type I boundary pivot and the collision from $\V$ is type (d), then 
all the boundary values at $\L(\otheta)$ are positive and $(\K'_n,\J'_n)_{n \in \{0,N'+1\}} = (\K_n, \J_n)_{n \in \{0, N+1\}}$ is unique, and there are exactly two internal sequences $(\K_n,\J_n)_{n=1}^{N}$ and $(\K'_n,\J'_n)_{n=1}^{N'}$ optimal at $\L(\otheta)$.  Hence we are in case (1).

If the pivot is a type I boundary pivot and the collision from $\V$ is type (e) or (f), 
Then the solution at $\L(\otheta)$ is unique, and away from $\L(\otheta)$ either $\bv>0$ or $\bv^*>0$ determine two unique boundary bases $(\K_n, \J_n)_{n \in \{0, N+1\}}$ and $(\K'_n,\J'_n)_{n \in \{0,N'+1\}}$.  Further more for each or those there is a unique  $(\K_n,\J_n)_{n=1}^{N}$ and  $(\K'_n,\J'_n)_{n=1}^{N'}$, so again we are in case (1).

Consider now the case of a type II pivot, in which $\bv=0$ at the collision and is  $>0$ on the $\V$ side, and it leaves the primal basis and its dual $\bv^*>0$ enters the dual basis, and $\bw^*>0$ leaves the dual basis and its dual $\bw=0$ enters the primal basis.  Then in $(\K_n, \J_n)_{n \in \{0, N+1\}}$ we have $\bv=0$, $\bw^*>0$, and in $(\K'_n, \J'_n)_{n \in \{0, N'+1\}}$ we have  
 $\bv^*>0$, $\bw=0$.  Apart from these two solutions of the Boundary-LP/LP$^*$, the only other solutions will be convex combinations of these two solutions, for which there is another base sequence $(\K''_n, \J''_n)_{n \in \{0, N''+1\}}$ in which the variables  $\bv^*, \bw^*$ are positive, and at the point $\L(\otheta)$ they can have values on a whole interval.  Hence for the base sequence $(\K''_n, \J''_n)_{n=0}^{N''+1}$ $\bH_\setP$ at the point $\L(\otheta)$ is not unique, and hence the matrix $M$ is singular. This implies by Corollary \ref{thm.nonsingular}  that the interior of the validity region of the sequence $(\K''_n, \J''_n)_{n=0}^{N''+1}$ is empty.
 \rule{0.5em}{0.5em} \end{proof}

\subsection{Proof of Theorem \ref{thm.alg-line}}
\label{sec.alg-line}
\begin{proof}
(i) By definition $\L(0), \L(1) \in \F$ and the parametric-feasible region is a closed convex polyhedral cone. 

(ii) If  $\L(0)$ is an interior point of a validity region then it is also an interior point of the parametric-feasible region.  By convexity of $\F$ its interior is convex, hence, if point $\L(1)$ is an interior point of the parametric-feasible region then the whole line $\L(\theta)$ belongs to the interior of the parametric-feasible region, otherwise all points of $\L(\theta)$ except $\L(1)$ are interior points of the parametric-feasible region.

(iii)  The parametric-feasible region is the union of validity regions of proper base sequences.  Hence $\L(\theta)$ is a union of intervals belonging to these validity regions.   Because validity regions are convex, no two intervals can belong to the same proper base sequence.  By Proposition \ref{thm.number} the number of proper base sequences is bounded by $\binom{4(K+J)} {2(K+J)}\; 2^{\binom{K+J}{K}}$.  Hence the number of intervals is finite.  It is then possible to choose a minimal number of proper base sequences which define a minimal number of intervals of positive length.  

Note, if part of the parametric line lies on a face of a validity region, then that interval belongs to more than one validity region.  In that case this interval will correspond to more than one proper base sequence.
\qquad
 \rule{0.5em}{0.5em} \end{proof}


\end{document}